\documentclass{article}

\usepackage{times}
\usepackage{epsfig}
\usepackage{graphicx}
\usepackage{amsmath,amsthm}
\usepackage{amssymb}
\usepackage[top=2cm, bottom=2cm, left=2.5cm, right=2.5cm]{geometry}

\usepackage[boxed]{algorithm2e}
\SetKw{Break}{break}
\usepackage{mathtools}
\usepackage{booktabs}
\usepackage{multirow}

\newcommand{\dx}{\,\mathrm{d}}
\newcommand{\C}{\mathbb{C}}
\newcommand{\R}{\mathbb{R}}
\newcommand{\N}{\mathbb{N}}
\newcommand{\Tmax}{T_\mathrm{max}}
\newcommand{\Id}{\mathrm{Id}}
\newcommand{\loss}{l}
\newcommand{\init}{\mathrm{init}}
\newcommand{\regularizerVector}{r}

\DeclareMathOperator*{\argmin}{argmin}

\theoremstyle{definition}
\newtheorem{definition}{Definition}[section]

\theoremstyle{plain}
\newtheorem{theorem}{Theorem}[section]

\theoremstyle{remark}

\usepackage[breaklinks=true]{hyperref}

\begin{document}

\title{Total Deep Variation for Linear Inverse Problems}

\author{Erich Kobler\\
Graz University of Technology\\
{\tt\small erich.kobler@icg.tugraz.at}
\and
Alexander Effland\\
Graz University of Technology\\
{\tt\small effland@tugraz.at}
\and
Karl Kunisch\\
University of Graz\\
{\tt\small karl.kunisch@uni-graz.at}
\and
Thomas Pock\\
Graz University of Technology\\
{\tt\small pock@icg.tugraz.at}
}

\maketitle

\begin{abstract}
Diverse inverse problems in imaging can be cast as variational problems composed of a task-specific data fidelity term and a regularization term.
In this paper, we propose a novel learnable general-purpose regularizer exploiting recent architectural design patterns from deep learning.
We cast the learning problem as a discrete sampled optimal control problem, for which we derive the adjoint state equations and an optimality condition.
By exploiting the variational structure of our approach, we perform a sensitivity analysis with respect to the learned parameters obtained from different training datasets.
Moreover, we carry out a nonlinear eigenfunction analysis, which reveals interesting properties of the learned regularizer.
We show state-of-the-art performance for classical image restoration and medical image reconstruction problems.
\end{abstract}

\section{Introduction}
The statistical interpretation of linear inverse problems allows the treatment of measurement uncertainties in the input data~$z$ and missing information in a rigorous framework.
Bayes' theorem states that the posterior distribution $p(x\vert z)$ is proportional to the product of the data likelihood~$p(z\vert x)$ and the prior~$p(x)$, which
represents the belief in a certain solution~$x$ given the input data~$z$.
A classical estimator for~$x$ is given by the maximum a posterior (MAP) estimator, which in a negative log-domain amounts to minimizing the variational problem
\begin{equation}
\mathcal{E}(x)\coloneqq\mathcal{D}(x,z)+\mathcal{R}(x).
\end{equation}
Here, the data fidelity term~$\mathcal{D}$ can be identified with the negative log-likelihood~$-\log p(z\vert x)$ and the regularization term corresponds to the negative log-probability of the prior distribution$-\log p(x)$.
Assuming Gaussian noise in the data~$z$, the data fidelity term naturally arises from the negative log-Gaussian, which essentially leads to a quadratic $\ell^2$-term of the form $\mathcal{D}(x,z)\coloneqq\frac{1}{2}\Vert Ax-z\Vert_2^2$.
In this paper, we assume that $A$ is a task-specific linear operator (see~\cite{ChPo16}) such as a downsampling operator in the case of single image super-resolution.
While the data fidelity term is straightforward to model, the grand challenge in inverse problems for imaging is the design of a regularizer that captures the complexity of the statistics of natural images.
\begin{figure}
\centering
\includegraphics[width=.4\linewidth]{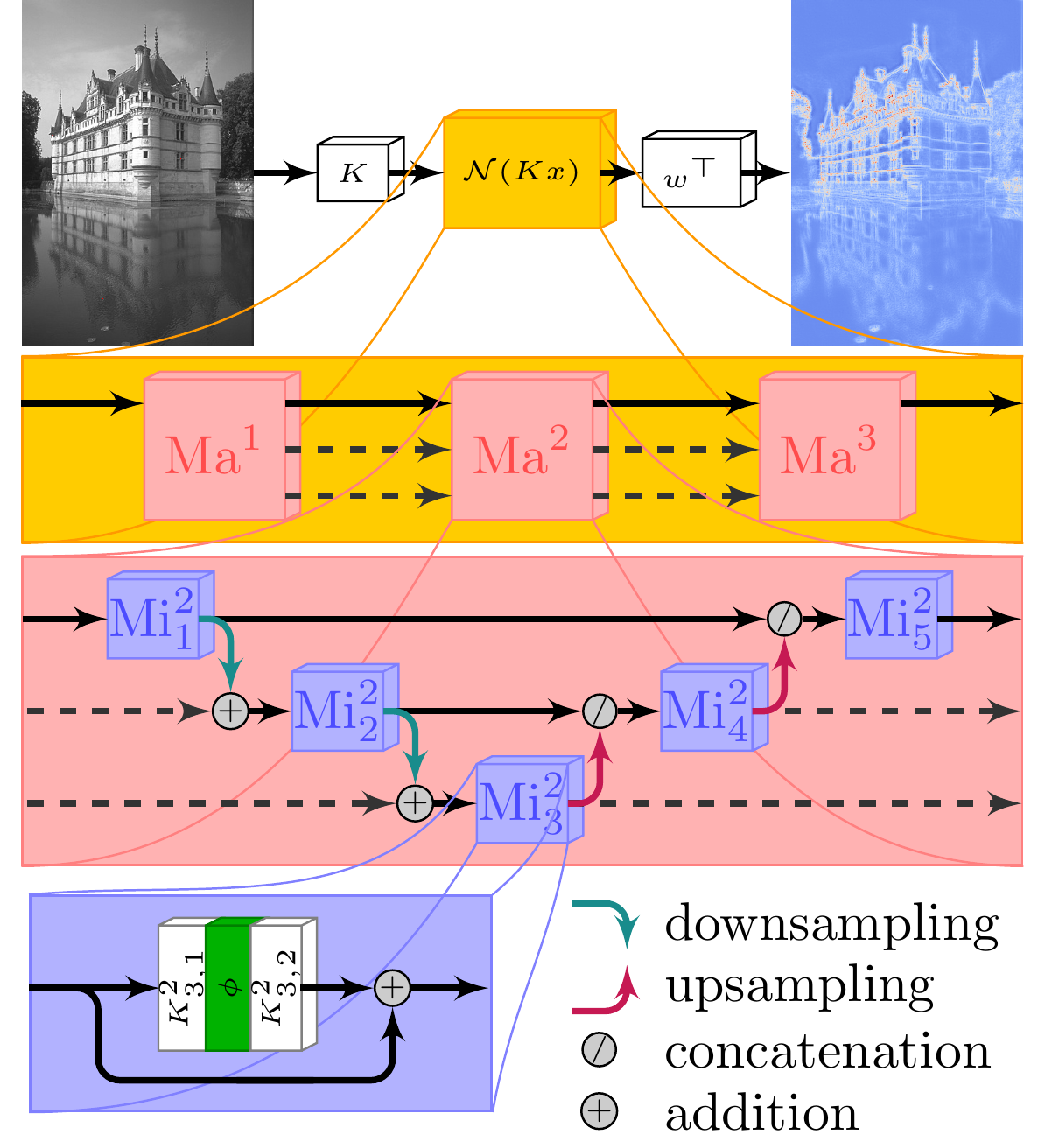}
\caption{
Visualization of the deep regularizer~$\regularizerVector$ (TDV$^3$), where colors indicate different hierarchical levels.    
On the highest level, $\regularizerVector(x,\theta)=w^\top \mathcal{N}(Kx)$ assigns to each pixel an energy value incorporating the local neighborhood.
The function~$\mathcal{N}$ (yellow) is composed of three macro-blocks (red), each representing a CNN with a U-Net type architecture.
The macro-blocks consist of five micro-blocks (blue) with a residual structure on three scales.
}
\label{fig:network}
\end{figure}

A classical and widely used regularizer is the total variation~(TV) originally proposed in~\cite{RuOsFa92}, which is based on the first principle assumption that images are piecewise constant with sparse gradients.
A well-known caveat of the sparsity assumption of TV is the formation of clearly visible artifacts known as staircasing effect.
To overcome this problem, the first principle assumption has later been extended to piecewise smooth images incorporating higher order image derivatives such as
infimal convolution based models~\cite{ChLi97} or the total generalized variation~\cite{BrKu10}.
Inspired by the fact that edge continuity plays a fundamental role in the human visual perception, regularizers penalizing the curvature of level lines have been proposed in~\cite{NiMu93,ChKa02,ChPo19}.
While these regularizers are mathematically well-understood, the complexity of natural images is only partially reflected in their formulation.
For this reason, handcrafted variational methods have nowadays been largely outperformed by purely data-driven methods.

It has been recognized quite early that a proper statistical modeling of regularizers should be based on learning~\cite{ZhWu98},
which has recently been advocated e.g.~in~\cite{LuOk18,LiSc20}.
One of the most successful early approaches is the Fields of Experts (FoE) regularizer~\cite{RoBl09}, which can be interpreted as a generalization of the total variation, but builds upon learned filters and learned potential functions.
While the FoE prior was originally learned generatively, it was shown in~\cite{SaTa09} that a discriminative learning via implicit differentiation yields improved performance.
A computationally more feasible method for discriminative learning is based on unrolling a finite number of iterations of a gradient descent algorithm~\cite{Do12}.
Additionally using iteration dependent parameters in the regularizer was shown to significantly increase the performance (TNRD~\cite{ChPo17}, \cite{Le16}).
In~\cite{KoKl17}, variational networks (VNs) are proposed, which give an incremental proximal gradient interpretation of TNRD.

Interestingly, such truncated schemes are not only computationally much more efficient, but are also superior in performance with respect to the full minimization.
A continuous time formulation of this phenomenon was proposed in~\cite{EfKo19} by means of an optimal control problem, within which an optimal stopping time is learned.

An alternative approach to incorporate a regularizer into a proximal algorithm, known as plug-and-play prior~\cite{VeSi13} or regularization by denoising~\cite{RoEl17},
is the replacement of the proximal operator by an existing denoising algorithm such as BM3D~\cite{DaFo07}.
Combining this idea with deep learning was proposed in~\cite{MeMo17,RiCh17}.
However, all the aforementioned schemes lack a variational structure and thus are not interpretable in the framework of MAP inference.

In this paper, we introduce a novel regularizer, which is inspired by the design patterns of state-of-the-art deep convolutional neural networks and simultaneously ensures a variational structure.
We achieve this by representing the total energy of the regularizer by means of a residual multi-scale network (Figure~\ref{fig:network}) leveraging smooth activation functions.
In analogy to~\cite{EfKo19}, we start from a gradient flow of the variational energy and utilize a semi-implicit time discretization, for which we derive the discrete adjoint state equation using the discrete Pontryagin maximum principle.
Furthermore, we present a first order necessary condition of optimality to automatize the computation of the optimal stopping time.
Our proposed \emph{Total Deep Variation} (TDV) regularizer can be used as a generic regularizer in variational formulations of linear inverse problems.
The major contributions of this paper are as follows:
\vspace{-1ex}
\begin{itemize} \setlength\itemsep{0em}
\item
The design of a novel generic multi-scale variational regularizer learned from data.
\item
A rigorous mathematical analysis including a sampled optimal control formulation of the learning problem
and a sensitivity analysis of the learned estimator with respect to the training dataset.
\item
A nonlinear eigenfunction analysis for the visualization and understanding of the learned regularizer.
\item
State-of-the-art results on a number of classical image restoration and medical image reconstruction problems with an impressively low number of learned parameters.
\end{itemize}
Due to space limitations, all proofs are presented in the appendix.

\section{Sampled optimal control problem}
Let $x\in\R^{nC}$ be a corrupted input image with a resolution of $n=n_1\cdot n_2$ and $C$~channels.
We emphasize that the subsequent analysis can easily be transferred to image data in any dimension.

The variational approach for inverse problems frequently amounts to computing a minimizer of a specific energy functional~$\mathcal{E}$ of the form $\mathcal{E}(x,\theta,z)\coloneqq\mathcal{D}(x,z)+\mathcal{R}(x,\theta)$.
Here, $\mathcal{D}$ is a data fidelity term and a $\mathcal{R}$ is a parametric regularizer depending on the learned training parameters~$\theta\in\Theta$,
where $\Theta\subset\R^p$ is a compact and convex set of training parameters.
Typically, the minimizer of this variational problem is considered as an approximation of the uncorrupted ground truth image.
Accordingly, in this paper we analyze different data fidelity terms of the form
$\mathcal{D}(x,z)=\frac{1}{2}\Vert Ax-z\Vert_2^2$ for fixed task-dependent $A\in\R^{lC\times nC}$ and fixed $z\in\R^{lC}$.

The proposed total deep variation~(TDV) regularizer is given by $\mathcal{R}(x,\theta)=\sum_{i=1}^n r(x,\theta)_i\in\R$, 
which is the total sum of the pixelwise deep variation defined as $\regularizerVector(x,\theta)=w^\top\mathcal{N}(Kx)\in\R^n$.
Here, $K\in\R^{nm\times nC}$ is a learned convolution kernel with zero-mean constraint (i.e. $\sum_{i=1}^{nC}K_{j,i}=0$ for $j=1,\ldots,nm$),
$\mathcal{N}:\R^{nm}\to\R^{nq}$ is a multiscale convolutional neural network and $w\in\R^q$ is a learned weight vector.
Hence, $\theta$ encodes the kernel~$K$, the weights of the convolutional layers in~$\mathcal{N}$ and the weight vector~$w$.
The exact form of the network is depicted in Figure~\ref{fig:network}.
In the TDV$^l$ network for $l\in\N$, the function~$\mathcal{N}$ is composed of~$l$ consecutive macro-blocks~$\mathbf{Ma}^1,\ldots,\mathbf{Ma}^l$ (red blocks).
Each macro-block~$\mathbf{Ma}^i$ ($i\in\{1,\ldots,l\}$) has a U-Net type architecture~\cite{RoFi15} with five micro-blocks $\mathbf{Mi}_1^i,\ldots,\mathbf{Mi}_5^i$ (blue blocks) distributed over three scales with skip-connections on each scale.
In addition, residual connections are added between different scales of consecutive macro-blocks whenever possible.
Finally, each micro-block~$\mathbf{Mi}_j^i$ for $i\in\{1,\ldots,l\}$ and $j\in\{1,\ldots,5\}$ has a residual structure defined as
$\mathbf{Mi}_j^i(x)=x+K_{j,2}^i\phi(K_{j,1}^i x)$ for convolution operators~$K_{j,1}^i$ and~$K_{j,2}^i$.
We use a smooth log-student-t-distribution of the form $\phi(x)=\frac{1}{2\nu}\log(1+\nu x^2)$, which is an established model for the statistics of natural images~\cite{HuMu99} and has the properties $\phi'(0)=0$ and $\phi''(0)=1$.
All convolution operators $K_{j,1}^i$ and $K_{j,2}^i$ in each micro-block correspond to $3\times3$ convolutions with $m$~feature channels and no additional bias.
To avoid aliasing, downsampling and upsampling are implemented incorporating $3\times 3$ convolutions and transposed convolutions with stride~$2$ using a blurring of the kernels proposed by~\cite{Zh19}.
Finally, the concatenation maps $2m$~feature channels to $m$~feature channels using a single $1\times 1$ convolution.

A common strategy for the minimization of the energy~$\mathcal{E}$ is the incorporation of
a \emph{gradient flow}~\cite{AmGi08} posed on a finite time interval~$(0,T)$, which reads as
\begin{align}
\dot{\tilde{x}}(t)=&-\nabla_1\mathcal{E}(\tilde{x}(t),\theta,z)=f(\tilde{x}(t),\theta,z)\label{eq:originalGradientFlow}\\
\coloneqq&-A^\top(A\tilde{x}(t)-z)-\nabla_1\mathcal{R}(\tilde{x}(t),\theta)
\end{align}
for $t\in(0,T)$, where $\tilde{x}(0)=x_\init$ for a fixed \emph{initial value} $x_\init\in\R^{nC}$.
Here, $\tilde{x}:[0,T]\to\R^{nC}$ denotes the differentiable \emph{flow of $\mathcal{E}$} and $T>0$ refers to the \emph{time horizon}.
We assume that $T\in[0,\Tmax]$ for a fixed $\Tmax>0$.
To formulate optimal stopping, we exploit the reparametrization $x(t)=\tilde{x}(tT)$, which results for $t\in(0,1)$ in the equivalent gradient flow formulation
\begin{equation}
\dot{x}(t)=Tf(x(t),\theta,z),\qquad x(0)=x_\init.\label{eq:gradientFlow}
\end{equation}
Note that the gradient flow inherently implies a variational structure (cf.~\eqref{eq:originalGradientFlow}).

Following~\cite{EHa19,LiCh17,LiHa18}, we cast the training process as a \emph{sampled optimal control problem} with control parameters~$\theta$ and~$T$.
For fixed~$N\in\N$, let $(x_\init^i,y^i,z^i)_{i=1}^N\in(\R^{nC}\times\R^{nC}\times\R^{lC})^N$ be a collection of~$N$ triplets of initial-target image pairs, and observed data independently drawn from a task-dependent fixed probability distribution.
In additive Gaussian image denoising, for instance, a ground truth image~$y^i$ is deteriorated by noise~$n^i\sim\mathcal{N}(0,\sigma^2)$
and a common choice is $x_\init^i=z^i=y^i+n^i$.
Let $\loss:\R^{nC}\to\R_0^+$ be a convex, twice continuously differentiable and coercive (i.e.~$\lim_{\Vert x\Vert_2\to\infty}\loss(x)=+\infty$) function.
We will later use $\loss(x)=\sqrt{\Vert x\Vert_1^2+\varepsilon^2}$ for $\varepsilon>0$ and $\loss(x)=\frac{1}{2}\Vert x\Vert_2^2$.
Then, the sampled optimal control problem reads as
\begin{equation}
\inf_{T\in[0,\Tmax],\,\theta\in\Theta}
\left\{J(T,\theta)\coloneqq\frac{1}{N}\sum_{i=1}^N\loss(x^i(1)-y^i)\right\}
\label{eq:objectiveFunction}
\end{equation}
subject to the state equation for each sample
\begin{equation}
\dot{x}^i(t)=Tf(x^i(t),\theta,z^i),\qquad x^i(0)=x_\init^i\label{eq:gradientFlowSample}
\end{equation}
for $i=1,\ldots,N$ and $t\in(0,1)$.
The next theorem ensures the existence of solutions to this optimal control problem.
\begin{theorem}[Existence of solutions]
The minimum in~\eqref{eq:objectiveFunction} subject to the side conditions~\eqref{eq:gradientFlowSample} is attained.
\end{theorem}

\section{Discretized optimal control problem}
In this section, we present a novel fully discrete formulation of the previously introduced sampled optimal control problem.
The state equation~\eqref{eq:gradientFlowSample} is discretized using a semi-implicit scheme resulting in
\begin{equation}
x_{s+1}^i=x_s^i-\tfrac{T}{S}A^\top(Ax_{s+1}^i-z^i)-\tfrac{T}{S}\nabla_1\mathcal{R}(x_s^i,\theta)\in\R^{nC}
\label{eq:stateEquationDiscrete}
\end{equation}
for $s=0,\ldots,S-1$ and $i=1,\ldots,N$, where the \emph{depth}~$S\in\N$ is a priori fixed.
This equation is equivalent to $x_{s+1}^i=\widetilde{f}(x_s^i,T,\theta,z^i)$ with
\begin{equation}
\widetilde{f}(x,T,\theta,z)\coloneqq(\Id+\tfrac{T}{S}A^\top A)^{-1}(x+\tfrac{T}{S}(A^\top z-\nabla_1\mathcal{R}(x,\theta))).
\label{eq:discreteRightHandSide}
\end{equation}
The initial state satisfies $x_0^i=x_\init^i\in\R^{nC}$.
Then, the discretized sampled optimal control problem reads as
\begin{equation}
\inf_{T\in[0,\Tmax],\,\theta\in\Theta}
\left\{J_S(T,\theta)\coloneqq\frac{1}{N}\sum_{i=1}^N\loss(x_S^i-y^i)\right\}
\label{eq:discreteOptimalControl}
\end{equation}
subject to $x_{s+1}^i=\widetilde{f}(x_s^i,T,\theta,z^i)$.
Following the discrete Pontryagin maximum principle~\cite{Ha66,LiHa18}, the associated discrete adjoint state~$p_s^i$ is given by
\begin{equation}
p_s^i=(\Id-\tfrac{T}{S}\nabla_1^2\mathcal{R}(x_s^i,\theta))(\Id+\tfrac{T}{S}A^\top A)^{-1}p_{s+1}^i
\label{eq:adjointEquationDiscrete}
\end{equation}
for $s=S-1,\ldots,0$ and $i=1,\ldots,N$, and the terminal condition $p_S^i=-\frac{1}{N}\nabla\loss(x_S^i-y^i)$.
For further details, we refer the reader to the appendix.

The next theorem states an exactly computable condition for the optimal stopping time, which is of vital importance for the numerical optimization:
\begin{theorem}[Optimality condition]
Let $(\overline{T},\overline{\theta})$ be a stationary point of~$J_S$ with associated states~$\overline{x}_s^i$ and adjoint states $\overline{p}_s^i$ satisfying~\eqref{eq:stateEquationDiscrete} and \eqref{eq:adjointEquationDiscrete}
subject to the initial conditions $\overline{x}_0^i=x_\init^i$ and the terminal conditions $\overline{p}_S^i=-\frac{1}{N}\nabla\loss(\overline{x}_S^i-y^i)$.
We further assume that $\nabla \widetilde{f}(\overline{x}_s^i,\overline{T},\overline{\theta},z^i)$ has full rank for all $i=1,\ldots,N$ and $s=0,\ldots,S$.
Then,
\begin{equation}
-\frac{1}{N}\sum_{s=0}^{S-1}\sum_{i=1}^N\langle\overline{p}_{s+1}^i,(\Id+\tfrac{\overline{T}}{S}A^\top A)^{-1}(\overline{x}_{s+1}^i-\overline{x}_s^i)\rangle=0.
\label{eq:optimalTDiscrete}
\end{equation}
\end{theorem}
Note that~\eqref{eq:optimalTDiscrete} is the derivative of the discrete Lagrange functional minimizing~$J_S$ subject to the discrete state equation.
The proof and a time continuous version of this theorem are presented in the appendix.

An important property of any learning based method is the dependency of the learned parameters with respect to different training datasets~\cite{GeBi92}.
\begin{theorem}\label{thm:sensitivity}
Let $(T,\theta),(\widetilde{T},\widetilde{\theta})$ be two pairs of control parameters obtained from two different training datasets.
We denote by~$x,\widetilde x\in(\R^{nC})^{(S+1)}$ two solutions of the state equation with the same observed data~$z$ and initial condition~$x_\init$, i.e.
\begin{equation}
x_{s+1}=\widetilde{f}(x_s,T,\theta,z),\quad
\widetilde{x}_{s+1}=\widetilde{f}(\widetilde{x}_s,\widetilde{T},\widetilde{\theta},z)
\end{equation}
for $s=1,\ldots,S-1$ and $x_0=\widetilde{x}_0=x_\init$.
Let $B(T)\coloneqq\Id+\frac{T}{S}A^\top A$ and ~$L_\mathcal{R}$ be the Lipschitz constant of~$\mathcal{R}$, i.e.
\begin{equation}
\Vert\nabla_1\mathcal{R}(x,\theta)-\nabla_1\mathcal{R}(\widetilde{x},\widetilde{\theta})\Vert_2\leq L_\mathcal{R}
\left\Vert
\begin{pmatrix}
x\\
\theta
\end{pmatrix}
-
\begin{pmatrix}
\widetilde{x}\\
\widetilde{\theta}
\end{pmatrix}
\right\Vert_2
\label{eq:Lipschitz}
\end{equation}
for all~$x,\widetilde{x}\in\R^{nC}$ and all~$\theta,\widetilde{\theta}\in\Theta$.
Then,
\begin{align}
\Vert x_{s+1}-\widetilde{x}_{s+1}\Vert_2
\leq&\Vert B(T)^{-1}-B(\widetilde{T})^{-1}\Vert_2
\Big(\Vert x_s\Vert_2+\tfrac{T}{S}\Vert A^\top z\Vert_2
+\tfrac{T}{S}\Vert\nabla_1\mathcal{R}(x_s,\theta)\Vert_2\Big)\notag\\
&+\Vert B(\widetilde{T})^{-1}\Vert_2
\Big(\Vert x_s-\widetilde{x}_s\Vert_2
+\tfrac{\vert T-\widetilde{T}\vert}{S}\Vert A^\top z\Vert_2+\tfrac{\vert T-\widetilde{T}\vert}{S}\Vert\nabla_1\mathcal{R}(\widetilde{x}_s,\widetilde{\theta})\Vert_2\notag\\
&+\tfrac{T}{S}L_\mathcal{R}\Vert(x,\theta)^\top-(\widetilde{x},\widetilde{\theta})^\top\Vert_2\Big).
\label{eq:sensitivity}
\end{align}
\end{theorem}
Hence, this theorem provides a computable upper bound for the norm difference of two states with observed value~$z$ and initial value~$x_\init$ evaluated at the same step~$s$,
which amounts to a sensitivity analysis w.r.t.~the training data.

\section{Numerical results}
In this section, we elaborate on the training and optimization, and we analyze the numerical results for four exemplary applications of TDVs: image denoising, CT and MRI reconstruction, and single image super-reconstruction.

\subsection{Training and optimization}
For all models considered, we solely use 400~images taken from the BSDS400 dataset~\cite{MaFo01} for training, where
we apply a data augmentation by flipping and rotating the images by multiples of $90^\circ$.
We compute the approximate minimizers of the discretized sampled optimal control problem~\eqref{eq:discreteOptimalControl} using the stochastic ADAM optimizer~\cite{KiBa15} on batches of size~$32$ with a patch size~$96\times 96$.
The zero-mean constraint of the kernel~$K$ is enforced by a projection after each iteration step.
For image denoising, we incorporate the squared $\ell^2$-loss function $\loss(x)=\frac{1}{2}\Vert x\Vert_2^2$ and the learning rate $4\cdot 10^{-4}$, whereas
for single image super-resolution the regularized $\ell^1$-loss $\loss(x)=\sqrt{\Vert x\Vert_1^2+\varepsilon^2}$ with $\varepsilon=10^{-3}$ and the learning rate $10^{-3}$ is used.
The first and second order momentum variables of the ADAM optimizer are set to $0.9$ and $0.999$ and $10^6$ training steps are performed.
Throughout all experiments we set $\nu=9$, the number of feature channels~$m=32$, and $S=10$ if not otherwise stated.

\subsection{Image denoising}
In the first task, we analyze the performance of the TDV regularizer for additive white Gaussian denoising.
To this end, we create the training dataset by uniformly drawing a ground truth patch~$y^i$ from the BSDS400 dataset and add Gaussian noise~$n^i\sim\mathcal{N}(0,\sigma^2)$.
Consequently, we set $x_\init^i=z^i=y^i+n^i$ and the linear operator~$A$ coincides with the identity matrix~$\Id$.

Table~\ref{tab:awgn} lists the average PSNR values of classical image test datasets for varying noise levels~$\sigma\in\{15,25,50\}$.
In the penultimate column, the PSNR values of our proposed TDV regularizer with three macro-blocks solely trained for~$\sigma=25$ (denoted by TDV$^3_{25}$) are presented.
To apply the TDV$^3_{25}$ model to different noise levels, we first rescale the noisy images~$\widehat{x}_\init^i=\widehat{z}^i=\tfrac{25}{\sigma}z^i$, then apply the learned scheme~\eqref{eq:stateEquationDiscrete}, and obtain the results via~$x_S^i=\tfrac{\sigma}{25}\widehat{x}_S^i$.
In the last column, the PSNR values of the proposed TDV regularizer with three macro-blocks (denoted by TDV$^3$) \emph{individually} trained for each specific noise levels are shown.
For all considered noise levels and datasets, state-of-the-art FOCNet~\cite{JiLi19} achieves the best results in terms of PSNR score.
The proposed TDV$^3_{25}$ and TDV$^3$ models have comparable PSNR values.
The adaption of the TDV$^3$ model to specific noise levels further increases the PSNR value.
To conclude, we emphasize that we achieve results on par with FOCNet with less than 1\% of the trainable parameters,
which highlights the potential of our approach.
Compared to FOCNet, the \emph{inherent variational structure} of our model allows for a \emph{deeper mathematical analysis}, that we elaborate on below.
\begin{table*}
\centering
\resizebox{\linewidth}{!}{
\begin{tabular}{l c*{7}{c} c}
\toprule[1.5pt]
Data set & $\sigma$ & BM3D~\cite{DaFo07} & TNRD~\cite{ChPo17} & DnCNN~\cite{ZhZu17} & FFDNet~\cite{ZhZu18} & N$^3$Net~\cite{PlRo18} & FOCNet~\cite{JiLi19} & TDV$^3_{25}$ & TDV$^3$ \\ \midrule[1pt]
\multirow{3}{*}{Set12} 
& 15 & 32.37 & 32.50 & 32.86 & 32.75 &  -    & 33.07 & 32.93 & 33.01\\
& 25 & 29.97 & 30.05 & 30.44 & 30.43 & 30.55 & 30.73 & 30.66 & 30.66\\
& 50 & 26.72 & 26.82 & 27.18 & 27.32 & 27.43 & 27.68 & 27.50 & 27.59\\ \midrule
\multirow{3}{*}{BSDS68} 
& 15 & 31.08 & 31.42 & 31.73 & 31.63 &  -    & 31.83 & 31.76 & 31.82\\
& 25 & 28.57 & 28.92 & 29.23 & 29.19 & 29.30 & 29.38 & 29.37 & 29.37\\
& 50 & 25.60 & 25.97 & 26.23 & 26.29 & 26.39 & 26.50 & 26.40 & 26.45\\ \midrule
\multirow{3}{*}{Urban100}
& 15 & 32.34 & 31.98 & 32.67 & 32.43 &  -    & 33.15 & 32.66 & 32.87 \\
& 25 & 29.70 & 29.29 & 29.97 & 29.92 & 30.19 & 30.64 & 30.38 & 30.38 \\
& 50 & 25.94 & 25.71 & 26.28 & 26.52 & 26.82 & 27.40 & 26.94 & 27.04 \\ \midrule[1pt]
\# Parameters
&    &       & 26,645 & 555,200 & 484,800 & 705,895 & 53,513,120 & 427,330 & 427,330\\
\bottomrule[1.5pt]
\end{tabular}
}
\caption{Comparison of average PSNR values for additive white Gaussian noise for~$\sigma\in\{15,25,50\}$ on classical image datasets.
In the last row, the number of trainable parameters is listed.}
\label{tab:awgn}
\end{table*}

Figure~\ref{fig:energyLandscape} depicts surface plots of the deep variation $[-1,1]\ni(\xi_1,\xi_2)\mapsto\regularizerVector(\xi_1 x+\xi_2 n)_i$
of TDV$_{25}^3$ evaluated at four prototypic patches~$x$ of size $49\times 49$, where $i$ is the index of the center pixel marked by the red points.
Here, $n$ refers to a Gaussian noise with standard deviation~$\sigma=25$.
Hence, the surface plots visualize the local regularization energy in the image contrast direction and a random noise direction.
The deep variation is smooth with commonly only a single local minimizer
and the shape significantly varies depending on the pixel neighborhood.
\begin{figure}
\centering
\includegraphics[width=.5\linewidth]{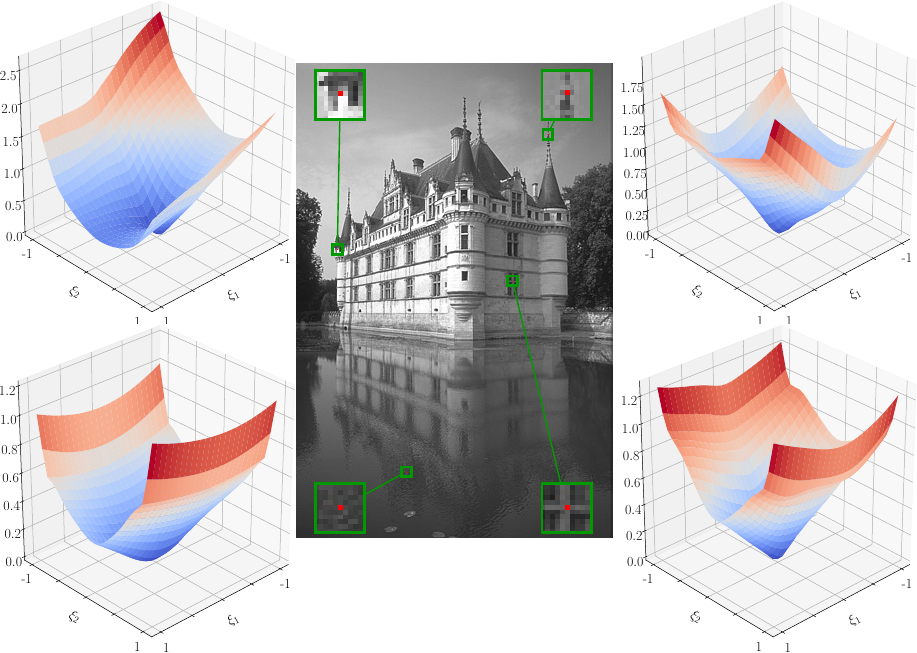}
\caption{Surface plots of the deep variation $[-1,1]\ni(\xi_1,\xi_2)\mapsto\regularizerVector(\xi_1 x+\xi_2 n)_i$
of four patches--each evaluated at the red center pixel using TDV$^3$ with BSDS68 and~$\sigma=25$.}
\label{fig:energyLandscape}
\end{figure}

Figure~\ref{fig:PSNRDenoising} (top) visualizes the average PSNR values $S\mapsto\frac{1}{N}\sum_{i=1}^N\mathrm{PSNR}(x_S^i,y^i)$ on the BSDS68 test dataset for $\sigma=25$.
The second plot $S\mapsto\overline{T}$ (bottom) shows the learned optimal stopping time as a function of the depth.
In all experiments, TDV regularizers with more macro-blocks perform significantly better.
Moreover, the average PSNR value saturates beyond the depth $S=10$.
The optimal stopping time converges for large~$S$ in all models.
Consequently, larger depth values~$S$ lead to a finer time discretization of the trajectories and depth values~$S$ exceeding 10~yield no further improvement.
\begin{figure}
\centering
\includegraphics[width=.5\linewidth]{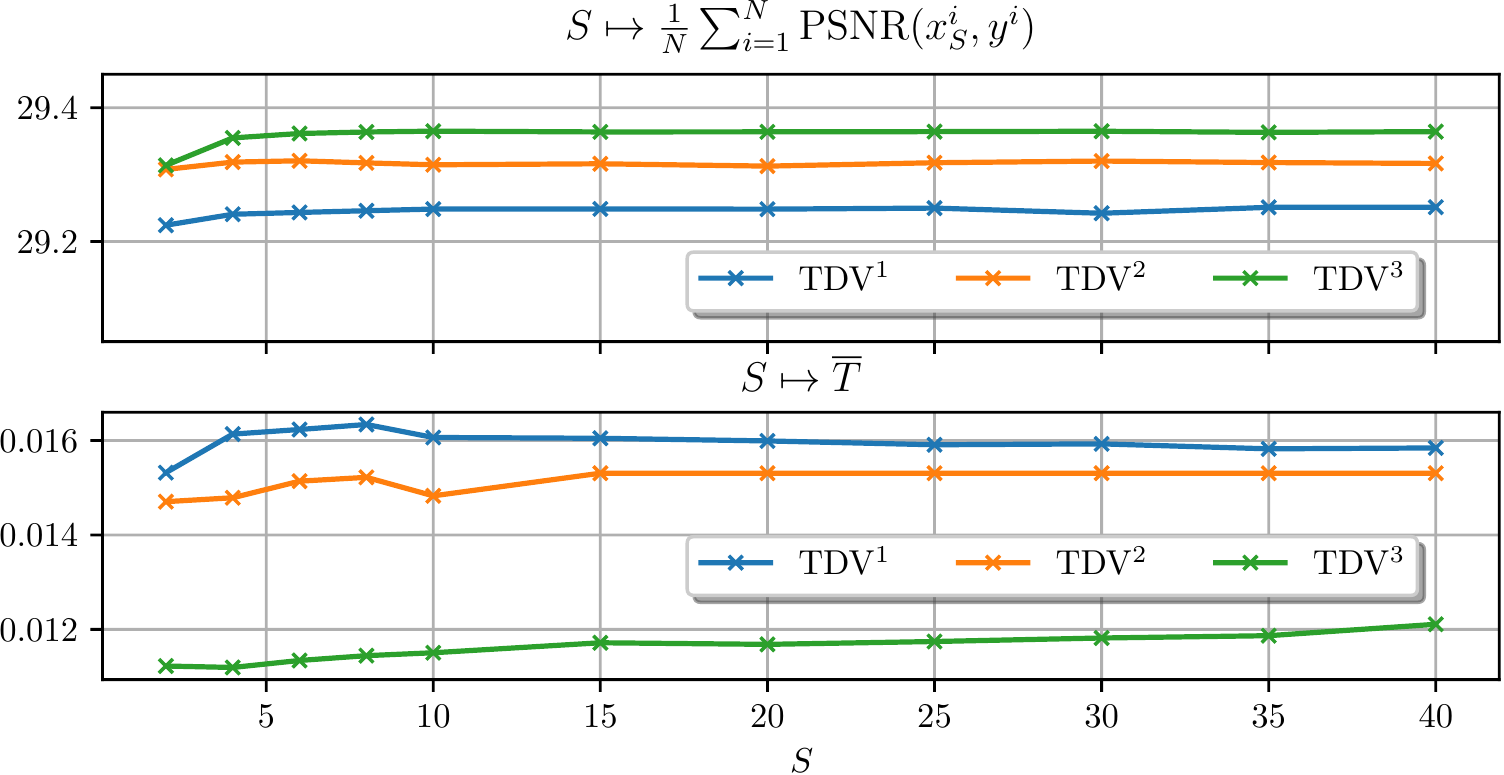}
\caption{Plots of the functions $S\mapsto\frac{1}{N}\sum_{i=1}^N\mathrm{PSNR}(x_S^i,y^i)$ (top) and $S\mapsto\overline{T}$ (bottom)
for TDV$^1$ (blue), TDV$^2$ (orange) and TDV$^3$ (green).}
\label{fig:PSNRDenoising}
\end{figure}
To address the significance of the optimal stopping time~$\overline{T}$, we evaluate the PSNR values (top) and the first order condition~\eqref{eq:optimalTDiscrete} (bottom) as a function of the stopping time in Figure~\ref{fig:OptimalStoppingTime} for TDV$_{25}^3$ trained for $S=10$.
Each black dashed curve represents a single image of the BSDS68 test dataset and the red curve is the average among all test samples.
Initially, all PSNR curves monotonically increase up to a unique maximum value located near the optimal stopping time~$\overline{T}=0.0297$ and decrease for larger values of~$T$,
which is exactly determined by the first order condition~\eqref{eq:optimalTDiscrete}.
We stress that all curves peak around a very small neighborhood, which results from an overlap of all curves near the zero crossing in the second plot.
\begin{figure}
\centering
\includegraphics[width=.5\linewidth]{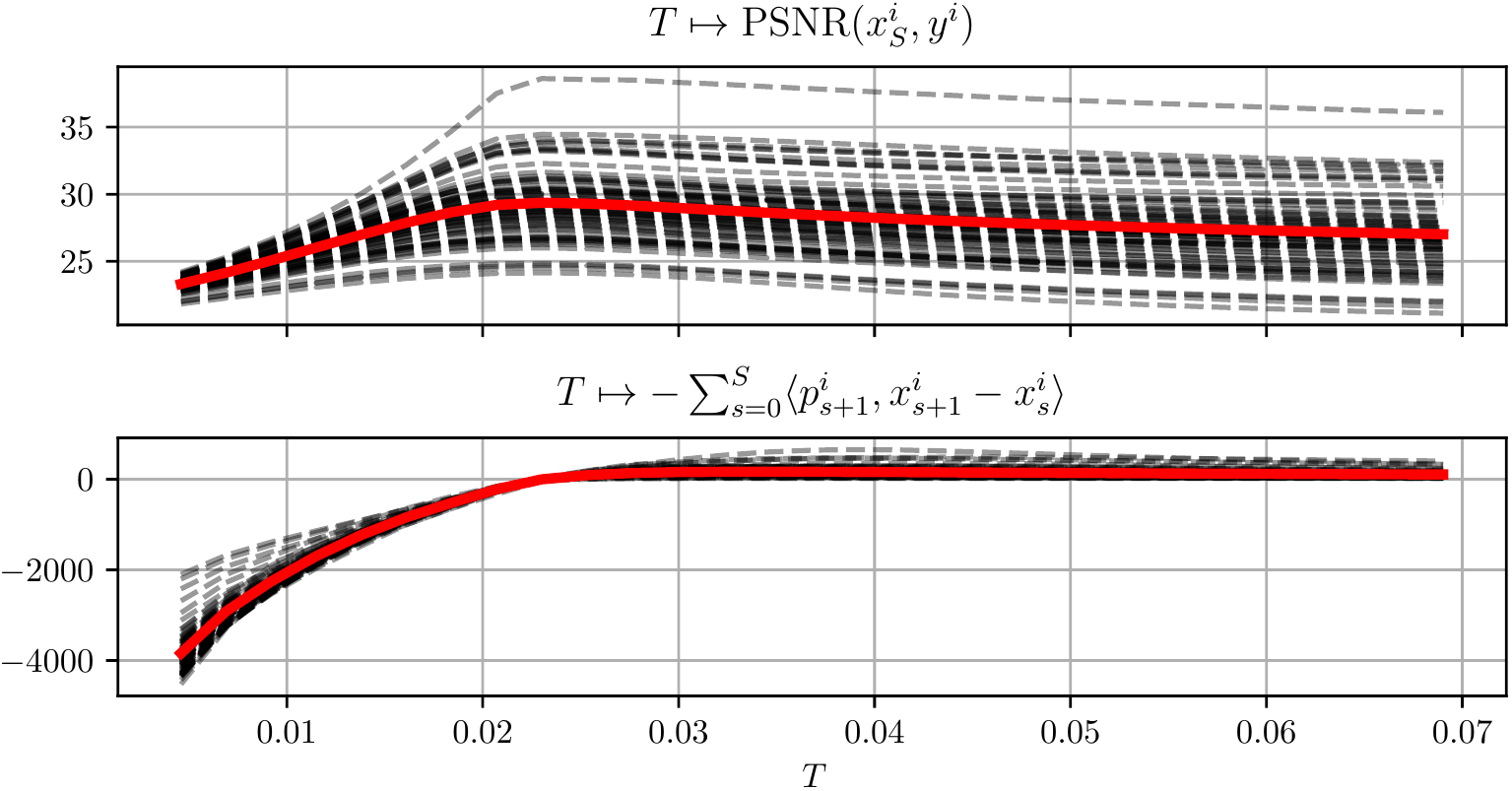}
\caption{Plots of the functions $T\mapsto\mathrm{PSNR}(x_S^i,y^i)$ (top) and the first order condition $T\mapsto-\sum_{s=0}^S \langle p_{s+1}^i,x_{s+1}^i-x_s^i\rangle$ (bottom) for $i=1,\ldots,68$ using TDV$^3$ for BSDS68 and~$\sigma=25$.
The averages across the samples are depicted by the red curves.}
\label{fig:OptimalStoppingTime}
\end{figure}

A qualitative and quantitative analysis of the impact of the stopping time for TDV$_{25}^3$ trained with $S=10$ on a standard test image is depicted in Figure~\ref{fig:OptimalStoppingSequence}.
Starting from the noisy input image, the restored image sequence for increasing $S\in\{5,10,15,20\}$ gradually removes noise.
Beyond the optimal stopping time, the model smoothes out fine details.
We emphasize that even high-frequency patterns like the vertical stripes on the chimney are consistently preserved.
Thus, the proposed model generates a stable and interpretable transition from noisy images to cartoon-like images
and the optimal stopping time determines the best intermediate result in terms of PSNR.
\begin{figure}
\includegraphics[width=\linewidth]{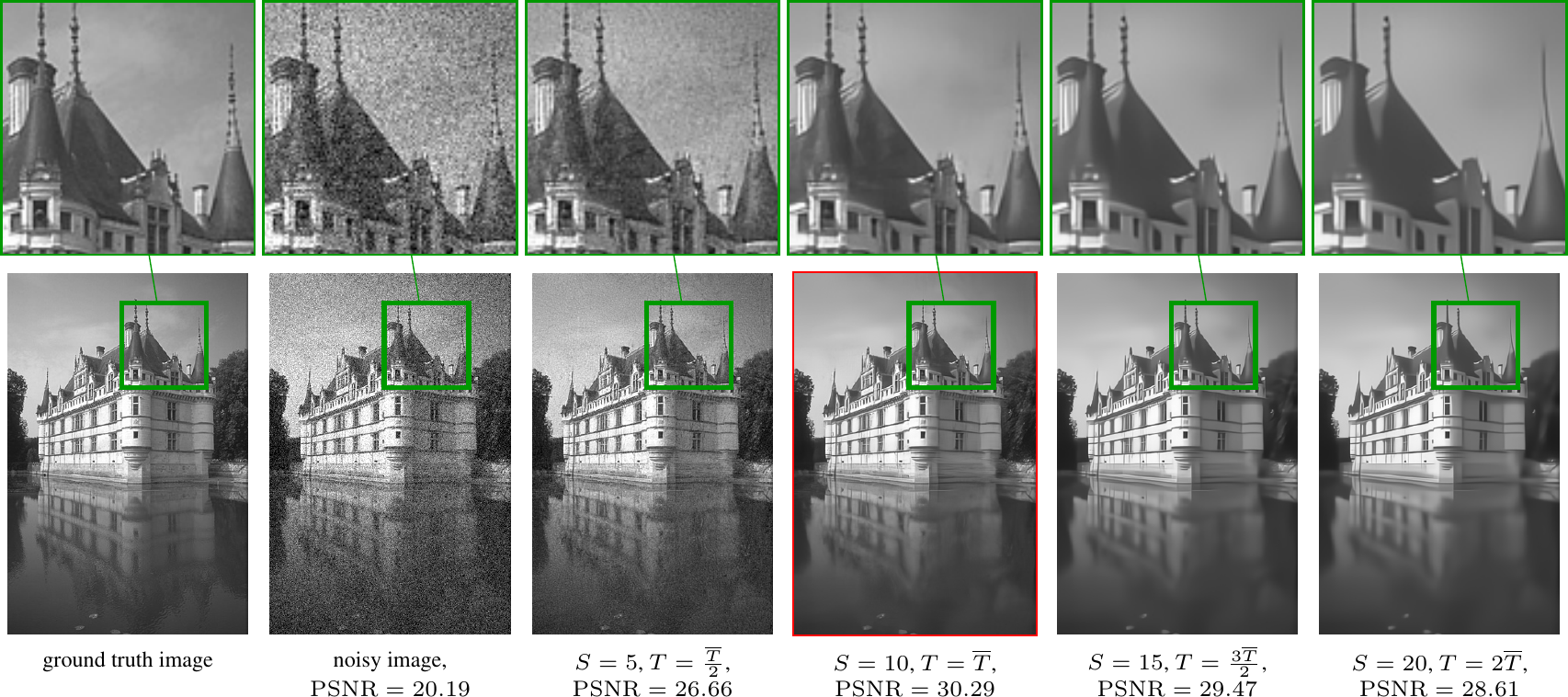}
\caption{From left to right: Ground truth, noisy input with noise level~$\sigma=25$ and resulting output of TDV$^3$ for $(S,T)\in\{(5,\frac{\overline{T}}{2}),(10,\overline{T}),(15,\frac{3\overline{T}}{2}),(20,2\overline{T})\}$,
where the optimal stopping time is $\overline{T}=0.0297$.
Note that the best image is framed in red.}
\label{fig:OptimalStoppingSequence}
\end{figure}

To analyze the local behavior of TDV$_{25}^3$, we compute a saddle point $(\overline{x},\overline{\lambda})$ of the Lagrangian
\begin{equation}
\mathcal{L}(x,\lambda)\coloneqq\mathcal{R}(x,\theta)-\frac{\lambda}{2}(\Vert x\Vert_2^2-\Vert x_\init\Vert_2^2),
\label{eq:constrainedOptimization}
\end{equation}
which is a optimization problem on the hypersphere $\Vert x\Vert_2=\Vert x_\init\Vert_2$ for a given input image~$x_\init\in\R^{nC}$ with Lagrange parameter~$\lambda$.
The optimality condition~\eqref{eq:constrainedOptimization} with respect to~$x$ is equivalent to $\nabla_1\mathcal{R}(\overline{x},\theta)=\overline{\lambda}\overline{x}$,
which shows that $(\overline{x},\overline{\lambda})$ is actually a nonlinear eigenpair of~$\nabla_1\mathcal{R}$.
Figure~\ref{fig:eigenpairs} depicts ten triplets of input images, eigenfunctions and eigenvalues~$(x_\init,\overline{x},\overline{\lambda})$, where the output images are computed using accelerated projected gradient descent.
As a result, the algorithm generates cartoon-like eigenfunctions and energy minimizing patterns are hallucinated due to the constraint.
For instance, novel stripe patterns are prolongated in the body and facial region of the woman in the third column or in the body parts of the zebra in the fifth column.
Furthermore, contours are axis aligned like the roof in the eighth column or the ship and the logo in the last image.
\begin{figure}
\includegraphics[width=\linewidth]{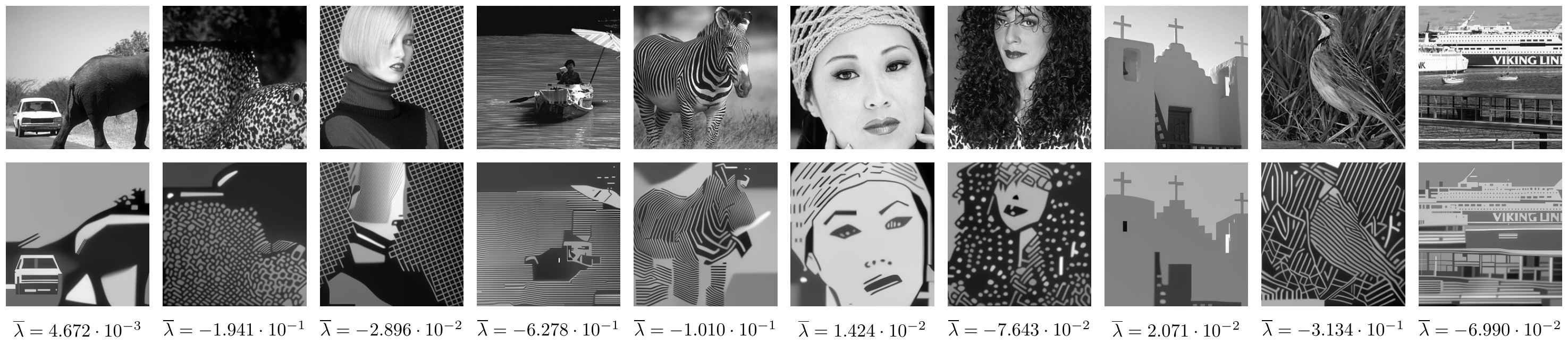}
\caption{Triplets of input images~$x_\init$ (top), eigenmodes~$\overline{x}$ (bottom) and eigenvalues~$\overline{\lambda}$ associated with the TDV$_{25}^3$ regularizer.}
\label{fig:eigenpairs}
\end{figure}

\begin{figure}
\centering
\includegraphics[width=.5\linewidth]{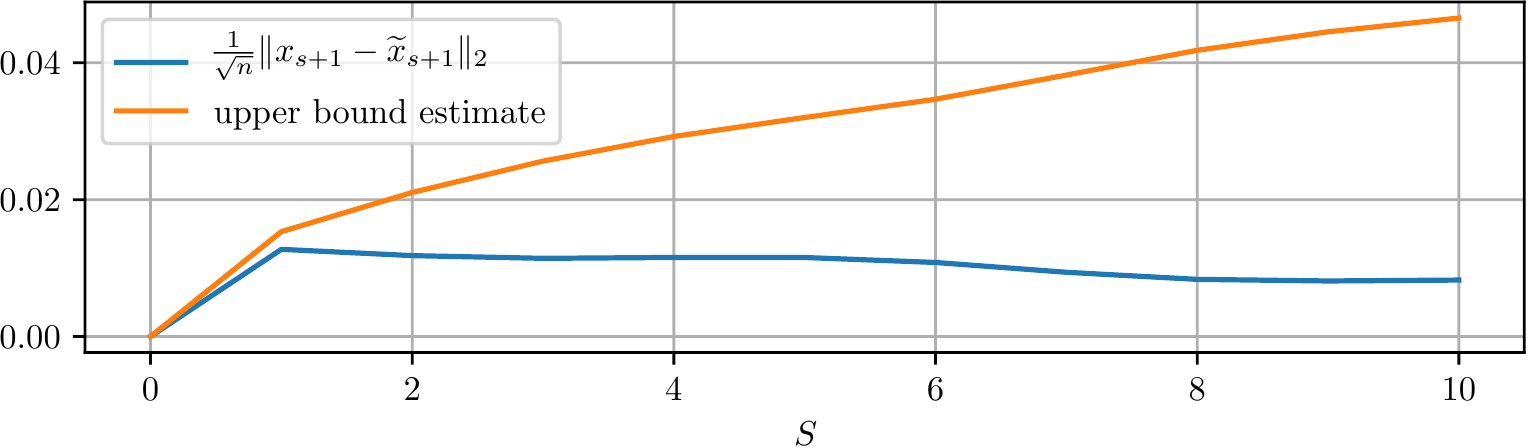}
\caption{
Comparison of RMSE along two sequences computed with two different parameter sets with the estimate~\eqref{eq:sensitivity}.
}
\label{fig:sensitivity}
\end{figure}   
In Figure~\ref{fig:sensitivity}, we compare the root mean squared error (RMSE) between the image sequences~$x_s$ and~$\widetilde{x}_s$
generated using TDV$^1$ with parameter sets~$(T,\theta)$ and $(\widetilde{T},\widetilde{\theta})$ with the upper bound estimate in Theorem~\ref{thm:sensitivity}.
Here, $(T,\theta)$ are computed by training on the entire BSDS400 training set,
whereas we solely use 10~randomly selected images of this training set to obtain~$(\widetilde{T},\widetilde{\theta})$.
The local Lipschitz constant~$L_\mathcal{R}$ in~\eqref{eq:Lipschitz} is estimated directly from the sequences.
As a result, the RMSE of the norm differences is roughly constant along~$S$ and the upper bound estimate is efficient.
Further details of the sensitivity analysis are presented in the appendix.

\subsection{Computed tomography reconstruction}
To demonstrate the broad applicability of the proposed TDV, we perform two-dimensional computed tomography (CT) reconstruction
using the TDV$^3_{25}$ regularizer trained for image denoising and $S=10$.
We stress that the regularizer is applied \emph{without} any additional training of the parameters.

The task of computed tomography is the reconstruction of an image given a set of projection measurements called sinogram, in which
the detectors of the CT scanner measure the intensity of attenuated X-ray beams.
Here, we use the linear attenuation model introduced in~\cite{HaMu18}, where the attenuation is proportional to the intersection area of a triangle, 
which is spanned by the X-ray source and a detector element, and the area of an image element.
In detail, the sinogram~$z$ of an image~$x$ is computed by~$z=A_R x$,
where $A_R$ is the lookup-table based area integral operator of~\cite{HaMu18} for $R$ angles and $768$ projections.
Typically, a fully sampled acquisition consists of $2304$~angles.
For this task, we consider the problem of angular undersampled CT~\cite{ChTa08}, where only a fraction of the angles are measured.
We use a 4-fold ($R=576$) and 8-fold ($R=288$) angular undersampling to reconstruct a representative image of the MAYO dataset~\cite{McBa17}.
To account for an imbalance of regularization and data fidelity, we manually scale the data fidelity term by~$\lambda>0$, i.e. $\mathcal{D}(x,z)\coloneqq\frac{\lambda}{2}\Vert A_Rx-z\Vert_2^2$.
The resulting smooth variational problem is optimized using accelerated gradient descent with Lipschitz backtracking.
Further details of the CT reconstruction task and the optimization scheme are included in the appendix.

We present qualitative and quantitative results for CT reconstruction in Figure~\ref{fig:CT} for a single abdominal CT image.
As an initialization, we perform $50$~steps of a conjugate gradient method on the data fidelity term (first and last column).
Using the proposed regularizer TDV$^3_{25}$, we are able to suppress the undersampling artifacts while preserving the fine vessels in the liver.
This highlights that the learned regularizer can be effectively applied as a generic regularizer for linear inverse problems without any transfer learning.

\begin{figure}
\includegraphics[width=\linewidth]{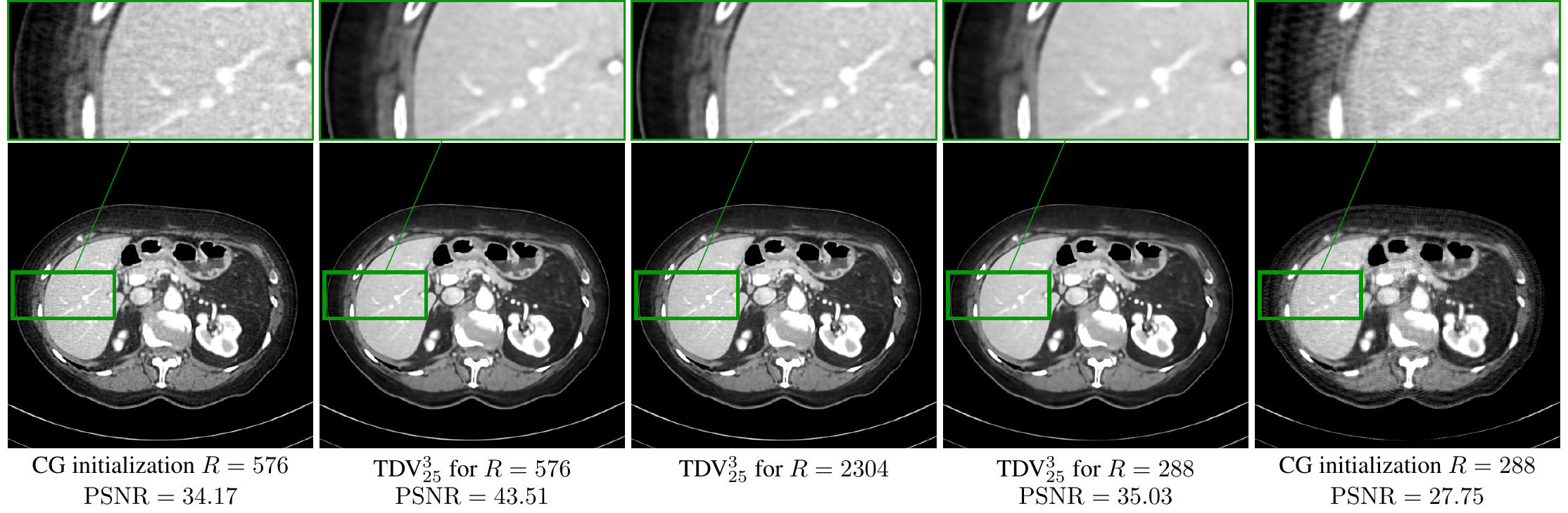}
\caption{Conjugate gradient reconstruction for 4/8-fold angular undersampled CT task (first/fifth image),
results obtained by using the TDV$^3_{25}$ regularizer with $\lambda=500\cdot10^3$ for 4-fold (second image) and $\lambda=10^6$ for 8-fold undersampling (fourth image), and fully sampled reference reconstruction using the TDV$^3_{25}$ regularizer with $\lambda=125\cdot10^3$ (third image).}
\label{fig:CT}
\end{figure}

\subsection{Magnetic resonance imaging reconstruction}
Next, the flexibility of our regularizer TDV$_{25}^3$ learned for denoising and $S=10$ is shown for accelerated magnetic resonance imaging (MRI)
\emph{without} any further adaption of~$\theta$.

In accelerated MRI, k--space data is acquired using $N_C$ parallel coils, each measuring a fraction of the full k--space to reduce acquisition time~\cite{HaKl18}.
Here, we use the data fidelity term
$\mathcal{D}(x,\{z_i\}_{i=1}^{N_C})=\frac{\lambda}{2}\sum_{i=1}^{N_C}\Vert M_RFC_ix-z_i\Vert_2^2$,
where $\lambda>0$ is a manually adjusted weighting parameter,
$M_R$ is a binary mask for $R$-fold undersampling, $F$ is the discrete Fourier transform, and $C_i$ are sensitivity maps computed following~\cite{UeLa14}.
We use 4-fold and 6-fold Cartesian undersampled MRI data to reconstruct a sample knee image.
Again, we minimize the resulting variational energy by accelerated gradient descent with Lipschitz backtracking.
Further details of this task and the optimization are in the appendix.

Figure~\ref{fig:MRI} depicts qualitative results and PSNR values for the reconstruction of 4-fold and 6-fold undersampled k--space data.
The first and last columns show the initial images obtained by applying the adjoint operator to the undersampled data.
Due to the undersampling in k--space both images are severely corrupted by backfolding artifacts,
which are removed by applying the proposed regularizer without losing fine details.

\begin{figure}
\includegraphics[width=\linewidth]{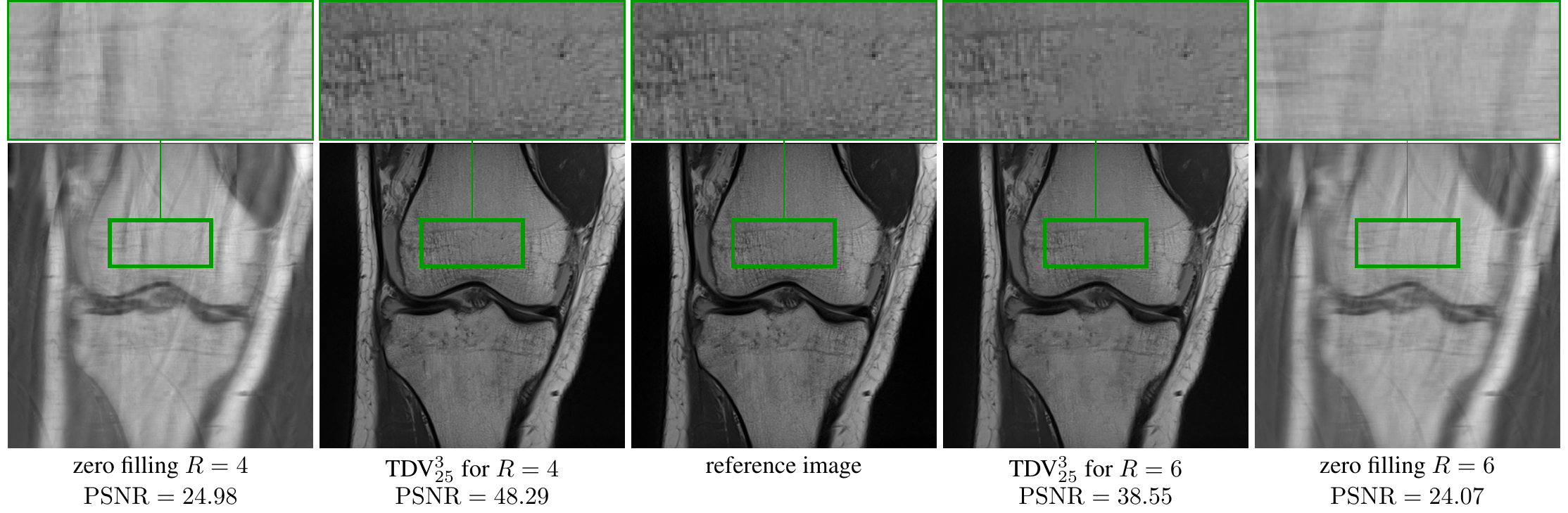}
\caption{Zero filling initialization for acceleration factors~$R\in\{4,6\}$ (first/fifth image), output using the TDV$^3_{25}$ regularizer with $\lambda=1000$ for $R=4$ (second image) and $\lambda=1500$ for $R=6$ (fourth image), and fully sampled reference (third image).}
\label{fig:MRI}
\end{figure}

\subsection{Single image super-resolution}

For single image super-resolution with scale factor $\gamma\in\{2,3,4\}$, we start with a full resolution ground truth image patch~$y^i\in\R^{nC}$ uniformly drawn from the BSDS400 dataset.
To obtain the linear downsampling operator~$A$, we implement the adjoint operator matching MATLAB\textsuperscript{\textregistered}'s bicubic upsampling operator \texttt{imresize},
which is an implementation of a scale factor-dependent interpolation convolution kernel in conjunction with a stride.
The observed low resolution image is $z^i=Ay^i\in\R^{nC/\gamma^2}$.
As an initialization, $3$~iteration steps of the conjugate gradient method for the variational problem
$x_\init^i=\argmin_{x\in\R^{nC}}\Vert Ax-z^i\Vert_2^2$ are applied.
The matrix inverse in~\eqref{eq:discreteRightHandSide} is approximated in each step by $7$~steps of a conjugate gradient scheme.
Here, all results are obtained by training a TDV$^3$ regularizer for each scale factor individually.

In Table~\ref{tab:SR}, we compare our approach with several state-of-the-art networks of similar complexity and list average PSNR values of the Y-channel in the YCbCr color space over test datasets.
For the BSDS100 dataset, our proposed method achieves similar results as OISR-LF-s~\cite{HeMo19} with only one third of the trainable parameters.
Figure~\ref{fig:OptimalStoppingSequenceSR} depicts a restored sequence of images for the single image super-resolution task with scale factor~$4$ using TDV$^3$ for a representative sample image of the Set14 dataset.
Starting from the low resolution initial image, interfaces are gradually sharpened and the best quality is achieved for~$T=\overline{T}$.
Beyond this point, interfaces are artificially intensified.
\begin{table*}
\centering
\resizebox{.8\linewidth}{!}{
\begin{tabular}{l c c*{3}{c} c}
\toprule[1.5pt]
Data set & Scale & MemNet~\cite{TaYa17} & VDSR~\cite{KiLe16} & DRRN~\cite{TaYa17a} & OISR-LF-s~\cite{HeMo19} & TDV$^3$ \\ \midrule[1pt]
\multirow{3}{*}{Set14} 
& $\times2$& 33.28 & 33.03 & 33.23 & 33.62 & 33.35\\
& $\times3$& 30.00 & 29.77 & 29.96 & 30.35 & 29.96\\
& $\times4$& 28.26 & 28.01 & 28.21 & 28.63 & 28.41\\ \midrule
\multirow{3}{*}{BSDS100} 
& $\times2$ & 32.08 & 31.90 & 32.05 & 32.20 & 32.18\\
& $\times3$ & 28.96 & 28.82 & 28.95 & 29.11 & 28.98\\
& $\times4$ & 27.40 & 27.29 & 27.38 & 27.60 & 27.50\\ \midrule[1pt]
\# Parameters & & 585,435 & 665,984 & 297,000& 1,370,000 & 428,970\\
\bottomrule[1.5pt]
\end{tabular}
}
\caption{Numerical results of various state-of-the art networks for single image super resolution with a comparable number of parameters.}
\label{tab:SR}
\end{table*}    

\begin{figure}
\includegraphics[width=\linewidth]{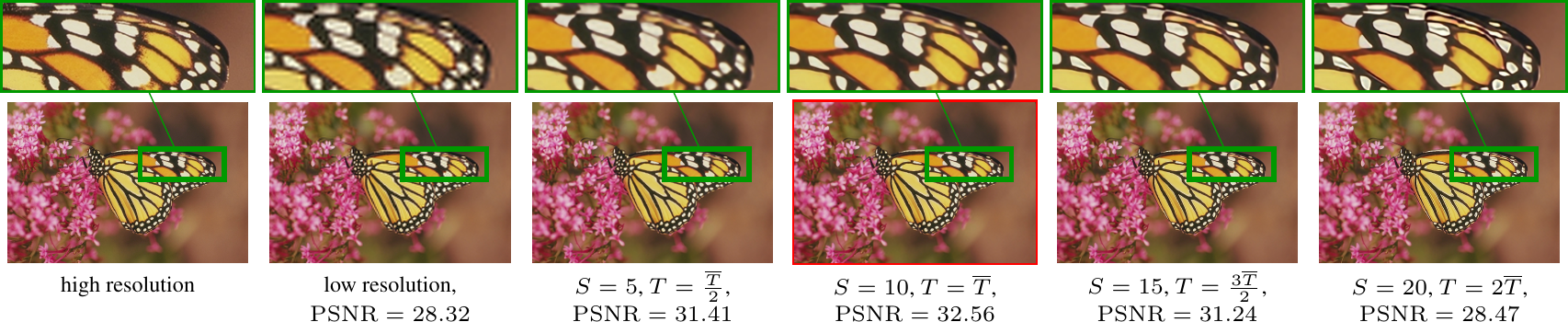}
\caption{
From left to right: High resolution, low resolution with scale factor~$4$, and resulting output of TDV$^3$ for $(S,T)\in\{(5,\frac{\overline{T}}{2}),(10,\overline{T}),(15,\frac{3\overline{T}}{2}),(20,2\overline{T})\}$,
where the optimal stopping time is $\overline{T}=0.098$.
Note that the best image is framed in red.
}
\label{fig:OptimalStoppingSequenceSR}
\end{figure}

\section{Conclusion}
In this paper, we have introduced total deep variation regularizers, which are motivated by established deep network architectures.
The inherent variational structure of our approach enables a rigorous mathematical understanding encompassing
an optimality condition for optimal stopping, a nonlinear eigenfunction analysis, and a sensitivity analysis that yields computable upper error bounds.
For image denoising and single image super-resolution, our model generates state-of-the-art results with an impressively low number of trainable parameters.
Moreover, to underline the versatility of TDVs for generic linear inverse problems, we successfully demonstrated their applicability for the challenging CT and MRI reconstruction tasks without requiring any additional training.

\subsection*{Acknowledgements}
\noindent
The authors thank Kerstin Hammernik for fruitful discussions and
acknowledge support from the ERC starting grant HOMOVIS (No. 640156) and ERC advanced grant OCLOC (No. 668998).

\appendix

\section{Notation and mathematical preliminaries}
In this section, we briefly introduce the notation and give a short presentation of the mathematical preliminaries required for the proofs.
All results can be found in~\cite{AdFo03,Te12}.

Let $\Omega\subset\R^l$ be a bounded domain.
We denote by $C^0(\overline\Omega,\R^n)$ the \emph{set of continuous functions} and by $C^k(\overline\Omega,\R^n)$, $k\in\N$, the \emph{set of $k$~times continuously differentiable functions} mapping form~$\overline\Omega$ to~$\R^n$.
The corresponding norms are
\begin{align}
\Vert f\Vert_{C^0(\overline\Omega,\R^n)}&\coloneqq\max_{x\in\overline\Omega}\Vert f(x)\Vert_2,\notag\\
\Vert f\Vert_{C^k(\overline\Omega,\R^n)}&\coloneqq\sum_{\vert\alpha\vert\leq k}\Vert D^\alpha f\Vert_{C^0(\overline\Omega,\R^n)},
\end{align}
respectively.
Here, we set $D^\alpha f=(\frac{\partial^{\alpha_1}f}{\partial x_1^{\alpha_1}},\ldots,\frac{\partial^{\alpha_n}f}{\partial x_{n}^{\alpha_n}})^\top$
for a multi-index $\alpha=(\alpha_1,\ldots,\alpha_n)\in\N_0^n$ with $\vert\alpha\vert=\sum_{i=1}^n\alpha_i$.
We define the \emph{H\"older space} as 
\begin{equation}
C^{k,s}(\overline{\Omega},\R^n)\coloneqq\{f\in C^k(\overline{\Omega},\R^n):\Vert f\Vert_{C^{k,s}(\overline{\Omega},\R^n)}<\infty\}
\end{equation}
for $k\in\N$ and $s\in(0,1]$, where the H\"older norm is
\begin{equation}
\Vert f\Vert_{C^{k,s}(\overline{\Omega},\R^n)}\coloneqq\Vert f\Vert_{C^k(\overline{\Omega},\R^n)}+\sup_{\substack{x,y\in\overline{\Omega}\\x\neq y}}\frac{\Vert f(x)-f(y)\Vert_2}{\Vert x-y\Vert_2^s}.
\end{equation}
For $1\leq p<\infty$, the \emph{Lebesgue space}~$L^p(\Omega,\R^n)$ is the set of all measurable mappings from~$\Omega$ to~$\R^n$ such that $\Vert f\Vert_{L^p(\Omega,\R^n)}<\infty$, where
\begin{equation}
\Vert f\Vert_{L^p(\Omega,\R^n)}\coloneqq\left(\int_\Omega\Vert f(x)\Vert_2^p\dx x\right)^\frac{1}{p}.
\end{equation}
The \emph{Sobolev space}~$H^m(\Omega,\R^n)$ with $m\in\N$ is the set of all functions~$f\in L^2(\Omega,\R^n)$ such that for all $\alpha\in\N_0^n$ with $\vert\alpha\vert\leq m$
a function~$f^{(\alpha)}\in L^2(\Omega,\R^n)$ exists satisfying
\begin{equation}
\int_\Omega\langle D^\alpha\eta,f\rangle\dx x=(-1)^{\vert\alpha\vert}\int_\Omega\langle\eta,f^{(\alpha)}\rangle\dx x
\end{equation}
for all smooth functions $\eta:\Omega\to\R^n$ with compact support.
The space $H^m(\Omega,\R^n)$ is endowed with the norm
\begin{equation}
\Vert f\Vert_{H^m(\Omega,\R^n)}\coloneqq\sum_{\vert\alpha\vert\leq m}\Vert f^{(\alpha)}\Vert_{L^2(\Omega,\R^n)}.
\end{equation}
In what follows, we list several theorems for later reference:
\begin{theorem}\label{thm:maximumDomainExistence}
We assume that for $x_\init\in\R^n$, $g\in C^0(\R\times\R^n,\R^n)$ and each~$T>0$ there exist constants $C_1(T),C_2(T)\in(0,\infty)$ such that
\begin{equation}
\Vert g(t,x)\Vert_2\leq C_1(T)+C_2(T)\Vert x\Vert_2
\end{equation}
for all $(t,x)\in[0,T]\times\R^n$.
Then all solutions of the initial value problem $\dot{x}(t)=g(t,x)$ with $x(0)=x_\init$ are defined for all~$t\geq 0$.
\end{theorem}
\begin{proof}
See \cite[Theorem~2.17]{Te12}.
\end{proof}
\begin{theorem}[Gr\"onwall's inequality]\label{thm:Gronwall}
Suppose that the functions $y,\alpha,\beta\in C^0([0,T],\R)$ satisfy
\begin{equation}
y(t)\leq\alpha(t)+\int_0^t\beta(s)y(s)\dx s
\end{equation}
for all~$t\in[0,T]$. If $\alpha$ is monotonically increasing (i.e.~$\alpha(s)\leq\alpha(t)$ for $s\leq t\in[0,T]$) and $\beta\geq0$, then
\begin{equation}
y(t)\leq\alpha(t)\exp\left(\int_0^t\beta(s)\dx s\right)
\end{equation}
for all~$t\in[0,T]$.
\end{theorem}
\begin{proof}
See \cite[Lemma~2.7]{Te12}.
\end{proof}
Next, we present special cases of important results in functional analysis tailored to our particular application.
\begin{definition}
A sequence $\{g_k\}_{k\in\N}$ in $H^m(\Omega,\R^n)$ converges weakly to $g\in H^m(\Omega,\R^n)$ (denoted by $g_k\rightharpoonup g$) if for every element $g^\ast$ of the dual space of $H^m(\Omega,\R^n)$
$g^\ast(g_k)\to g^\ast(g)$ as~$k\to\infty$.
\end{definition}

\begin{theorem}\label{thm:compactness}
Let $I\subset\R$ be a bounded and open interval and $\{g_k\}_{k\in\N}$ in $H^m(I,\R^n)$ be a uniformly bounded sequence in $H^m(I,\R^n)$, i.e.~there exists
a constant $C>0$ such that $\Vert g_k\Vert_{H^m(I,\R^n)}<C$ for all~$k\in\N$.
Then there exists a subsequence $\{g_{k_l}\}_{l\in\N}$ of $\{g_k\}_{k\in\N}$ and $g\in H^m(I,\R^n)$ such that $g_{k_l}\rightharpoonup g$ in $H^m(I,\R^n)$ as~$l\to\infty$.
\end{theorem}
\begin{proof}
See \cite[Chapter~8]{Al16}.
\end{proof}
\begin{theorem}[Sobolev embedding theorem]\label{thm:SobolevEmbedding}
Let $I\subset\R$ be a bounded and open interval.
If $k,m\in\N$ and $\alpha\in(0,1)$ satisfy $m-\frac{1}{2}>k+\alpha$, then the operator~$\Id:H^m(I,\R^n)\to C^{k,\alpha}(\overline{I},\R^n)$ is continuous and compact.
In particular, for every~$g\in H^m(I,\R^n)$ there exists a function~$\hat{g}\in C^{k,\alpha}(\overline{I},\R^n)$ that coincides with~$g$ almost everywhere such that
\begin{equation}
\Vert \hat{g}\Vert_{C^{k,\alpha}(\overline{I},\R^n)}\leq C\Vert g\Vert_{H^m(I,\R^n)}.
\end{equation}
The constant~$C$ solely depends on $k$, $m$, $I$ and $\alpha$.
\end{theorem}
\begin{proof}
See \cite[Theorem~10.13]{Al16}.
\end{proof}

\section{Existence of solutions}
In this section, we briefly recall the functional analytic setting of the optimal control problem.
Afterwards, we present the proof for the existence of minimizers of the optimal control problem.

Throughout this section, we consider the data fidelity term $\mathcal{D}(x,z)=\frac{1}{2}\Vert Ax-z\Vert_2^2$ for a fixed task-dependent $A\in\R^{lC\times nC}$ and fixed $z\in\R^{lC}$.
The total deep variation~$\mathcal{R}$ is parametrized by~$\theta\in\Theta$, where $\Theta\subset\R^p$ is a compact and convex set of admissible training parameters.
We highlight that the building blocks of~$\mathcal{R}$ are convolutional layers at different scales and the smooth log-student-t-distribution of the form $\phi(x)=\frac{1}{2\nu}\log(1+\nu x^2)$ for~$\nu>0$.
In particular, this structure implies
\begin{align}
\Vert\nabla_1\mathcal{R}(x,\theta)\Vert_2&\leq C_\mathcal{R}(\theta,\nu)\Vert x\Vert_2
\label{eq:growthEstimateR}
\end{align}
for a constant~$C_\mathcal{R}(\theta,\nu)>0$ depending on~$\theta$ and~$\nu$.
We define the energy functional as
\begin{equation}
\mathcal{E}(x,\theta,z)\coloneqq\mathcal{D}(x,z)+\mathcal{R}(x,\theta).
\end{equation}
Then, the gradient flow for the finite time interval~$(0,T)$ is
\begin{align}
\dot{\tilde{x}}(t)=&-\nabla_1\mathcal{E}(\tilde{x}(t),\theta,z)=f(\tilde{x}(t),\theta,z)\\
\coloneqq&-A^\top(A\tilde{x}(t)-z)-\nabla_1\mathcal{R}(\tilde{x}(t),\theta)
\end{align}
for $t\in(0,T)$, where $\tilde{x}(0)=x_\init$ for a fixed initial value $x_\init\in\R^{nC}$.
Here, $\tilde{x}:[0,T]\to\R^{nC}$ denotes the differentiable flow of $\mathcal{E}$, and $T>0$ refers to the time horizon.
We assume that $T\in[0,\Tmax]$ for a fixed $\Tmax>0$.
To formulate optimal stopping, we exploit the reparametrization $x(t)=\tilde{x}(tT)$, which results for $t\in(0,1)$ in the equivalent gradient flow formulation
\begin{equation}
\dot{x}(t)=Tf(x(t),\theta,z),\qquad x(0)=x_\init.
\end{equation}
For fixed~$N\in\N$, let $(x_\init^i,y^i,z^i)_{i=1}^N\in(\R^{nC}\times\R^{nC}\times\R^{lC})^N$ be a collection of~$N$ triplets of initial-target image pairs, and observed data independently drawn from a task-dependent fixed probability distribution.
Let $\loss:\R^{nC}\to\R_0^+$ be a convex, twice continuously differentiable and coercive (i.e.~$\lim_{\Vert x\Vert_2\to\infty}\loss(x)=+\infty$) function.
Then, the sampled optimal control problem reads as
\begin{equation}
\inf_{T\in[0,\Tmax],\,\theta\in\Theta}
\left\{J(T,\theta)\coloneqq\frac{1}{N}\sum_{i=1}^N\loss(x^i(1)-y^i)\right\}
\label{eq:objectiveFunctionApp}
\end{equation}
subject to the state equation for each sample
\begin{equation}
\dot{x}^i(t)=Tf(x^i(t),\theta,z^i),\qquad x^i(0)=x_\init^i\label{eq:gradientFlowSampleApp}
\end{equation}
for $i=1,\ldots,N$ and $t\in(0,1)$.
The next theorem ensures the existence of solutions to this optimal control problem.
\begin{theorem}[Existence of solutions]
The minimum in~\eqref{eq:objectiveFunctionApp} subject to the side conditions~\eqref{eq:gradientFlowSampleApp} is attained.
\end{theorem}
\begin{proof}
Without restriction, we consider the case $N=1$ and omit the superscript.
Taking into account~\eqref{eq:growthEstimateR} we can estimate as follows:
\begin{equation}
\Vert Tf(x,\theta,z)\Vert_2
\leq T(\Vert A\Vert_2\Vert z\Vert_2)+T(\Vert A\Vert_2^2+C_\mathcal{R}(\theta,\nu))\Vert x\Vert_2.
\label{eq:upperEstimateRHS}
\end{equation}
Thus, Theorem~\ref{thm:maximumDomainExistence} implies that~$[0,1]$ is contained in the maximum domain of existence of the state equation~\eqref{eq:gradientFlowSampleApp} for any initial value~$x_\init\in\R^n$.
Let $(T_j,\theta_j)\in[0,\Tmax]\times\Theta$ be a minimizing sequence for~$J$ with an associated state~$x_j\in C^1([0,1],\R^n)$ such that~\eqref{eq:gradientFlowSampleApp} with $(T,\theta)$ replaced by $(T_j,\theta_j)$ holds true.
Due to the compactness of $[0,\Tmax]$ and~$\Theta$ we can deduce the existence of a subsequence (not relabeled) such that $(T_j,\theta_j)\to(T,\theta)\in[0,\Tmax]\times\Theta$.
Then, Theorem~\ref{thm:Gronwall} in conjunction with~\eqref{eq:upperEstimateRHS} implies the uniform boundedness of $\Vert x_j(t)\Vert_2$ and thus also
$\Vert\dot{x}_j(t)\Vert_2$ for all $t\in[0,1]$ and all $j\in\N$.
In particular, the sequence~$\{x_j\}_{j\in\N}$ is uniformly bounded in the Hilbert space $H^1((0,1),\R^n)$.
Thus, taking into account Theorem~\ref{thm:compactness} we can infer that a subsequence (not relabeled) converges weakly in $H^1((0,1),\R^n)$ to~$x\in H^1((0,1),\R^n)$.
Theorem~\ref{thm:SobolevEmbedding} ensures the compact embedding of $H^1((0,1),\R^n)$ into $C^0([0,1],\R^n)$,
which implies $x_j\to x$ in $C^0([0,1],\R^n)$ and $x(0)=x_\init$.
Since~$f$ is a smooth function, we can even conclude $x\in C^1([0,1],\R^n)$ due to the general theory of ordinary differential equations presented in~\cite[Chapter~I]{Ha80}.
Note that~$x$ is actually the unique solution to~\eqref{eq:gradientFlowSampleApp}.
This follows from the convergence of $x_j\to x$ in $C^0([0,1],\R^n)$ and the smoothness of~$f$.
Finally, the theorem follows from the continuity of~$\loss$, i.e.
\begin{equation}
\lim_{j\to\infty}J(T_j,\theta_j)=\lim_{j\to\infty}\loss(x_j(1)-y)=\loss(x(1)-y)=J(T,\theta),
\end{equation}
where we used the uniform convergence of $x_j\to x$ in $C^0([0,1],\R^n)$.
\end{proof}

\section{Discretized optimal control problem}
In this section, we present details for the fully discrete formulation of the sampled optimal control problem.
The state equation~\eqref{eq:gradientFlowSampleApp} is discretized using a semi-implicit scheme resulting in
\begin{equation}
x_{s+1}^i=x_s^i-\tfrac{T}{S}A^\top(Ax_{s+1}^i-z^i)-\tfrac{T}{S}\nabla_1\mathcal{R}(x_s^i,\theta)\in\R^{nC}
\label{eq:stateEquationDiscreteApp}
\end{equation}
for $s=0,\ldots,S-1$ and $i=1,\ldots,N$, where the \emph{depth}~$S\in\N$ is a priori fixed.
This equation is equivalent to $x_{s+1}^i=\widetilde{f}(x_s^i,T,\theta,z^i)$ with
\begin{align}
\widetilde{f}(x,T,\theta,z)&\coloneqq B(T)^{-1}(x+\tfrac{T}{S}(A^\top z-\nabla_1\mathcal{R}(x,\theta))),\notag\\
B(T)&\coloneqq\Id+\frac{T}{S}A^\top A.
\end{align}
The initial state satisfies $x_0^i=x_\init^i\in\R^{nC}$.
Then, the discretized sampled optimal control problem reads as
\begin{equation}
\inf_{T\in[0,\Tmax],\,\theta\in\Theta}
\left\{J_S(T,\theta)\coloneqq\frac{1}{N}\sum_{i=1}^N\loss(x_S^i-y^i)\right\}
\label{eq:discreteOptimalControlApp}
\end{equation}
subject to $x_{s+1}^i=\widetilde{f}(x_s^i,T,\theta,z^i)$.
We define the Lagrange functional
$L_S:(\R^{nC})^{N(S+1)}\times[0,\Tmax]\times\Theta\times(\R^{lC})^N\times(\R^{nC})^{N(S+1)}\to\R$
as follows:
\begin{equation}
(x,T,\theta,z,p)\mapsto J_S(T,\theta)+\sum_{i=1}^N\langle p_0^i,x_0^i-x_\init^i\rangle
+\sum_{s=0}^{S-1}\sum_{i=1}^N\langle p_{s+1}^i,x_{s+1}^i-\widetilde{f}(x_s^i,T,\theta,z^i)\rangle.
\end{equation}
The discretization of the associated adjoint equation is now implied by the discrete Pontryagin maximum principle:
\begin{theorem}
Let $(\overline{T},\overline{\theta})$ be a pair of optimal control parameters for~\eqref{eq:discreteOptimalControlApp} with the state equation~$\{\overline{x}_s^i\}_{s=0,\ldots,S}^{i=1,\ldots,N}$.
We define the Hamiltonian
\begin{align}
H&:\R^{nC}\times\R^{nC}\times[0,\Tmax]\times\Theta\times\R^{lC}\to\R,\notag\\
(x,p,T,\theta,z)&\mapsto\langle p,\widetilde{f}(x,T,\theta,z)\rangle.
\end{align}
If it is further assumed that $\nabla \widetilde{f}(\overline{x}_s^i,\overline{T},\overline{\theta},z^i)$ has full rank for all $i=1,\ldots,N$ and $s=0,\ldots,S$,
then there exists an adjoint process~$\{\overline{p}_s^i\}_{s=0,\ldots,S}^{i=1,\ldots,N}$ such that
\begin{align}
\overline{x}_{s+1}^i&=\nabla_2 H(\overline{x}_s^i,\overline{p}_{s+1}^i,\overline{T},\overline{\theta},z^i),\notag\\
\overline{x}_0^i&=x_\init^i,\notag\\
\overline{p}_s^i&=\nabla_1 H(\overline{x}_s^i,\overline{p}_{s+1}^i,\overline{T},\overline{\theta},z^i),\notag\\
\overline{p}_S^i&=-\frac{1}{N}\nabla\loss(\overline{x}_S^i-y^i)\label{eq:PMP}.
\end{align}
Finally, the solution is optimal in the sense that
\begin{equation}
\sum_{i=1}^N H(\overline{x}_s^i,\overline{p}_{s+1}^i,\overline{T},\overline{\theta},z^i)
\geq \sum_{i=1}^N H(\overline{x}_s^i,\overline{p}_{s+1}^i,T,\theta,z^i)
\end{equation}
for all $T\in[0,\Tmax]$ and $\theta\in\Theta$.
\end{theorem}
\begin{proof}
For a complete proof, see \cite[Theorem~3]{LiHa18} and \cite{Ha66}.
In what follows, we give a heuristic justification for the definition of the Hamiltonian.
To this end, we consider the Lagrange functional~$L_S$ with multipliers $\{p_s^i\}_{s=0,\ldots,S}^{i=1,\ldots,N}$:
\begin{align}
L_S(x,T,\theta,z,p)
=\frac{1}{N}\sum_{i=1}^N\loss(x_S^i-y^i)+\sum_{i=1}^N\langle p_0^i,x_0^i-x_\init^i\rangle
+\sum_{s=0}^{S-1}\sum_{i=1}^N\left(\langle p_{s+1}^i,x_{s+1}^i\rangle-H(x_s^i,p_{s+1}^i,T,\theta,z^i)\right).\notag
\end{align}
Taking the derivatives with respect to $\{x_s^i\}_{s=0,\ldots,S}^{i=1,\ldots,N}$ and $\{p_s^i\}_{s=0,\ldots,S}^{i=1,\ldots,N}$ yields~\eqref{eq:PMP}.
\end{proof}
Thus, \eqref{eq:PMP} readily implies
\begin{align}
\overline{p}_s^i&=\nabla_1 H(\overline{x}_s^i,\overline{p}_{s+1}^i,\overline{T},\overline{\theta},z^i)
=(\Id-\tfrac{\overline{T}}{S}\nabla_1^2\mathcal{R}(\overline{x}_s^i,\overline{\theta}))(\Id+\tfrac{\overline{T}}{S}A^\top A)^{-1}\overline{p}_{s+1}^i,\notag\\
\overline{p}_S^i&=-\frac{1}{N}\nabla\loss(\overline{x}_S^i-y^i).
\label{eq:adjointEquationDiscreteApp}
\end{align}
The next theorem states an exact computable condition for the optimal stopping time, which is of vital importance for the numerical optimization:
\begin{theorem}[Optimality condition]
Let $(\overline{T},\overline{\theta})$ be a stationary point of~$J_S$ with associated states~$\overline{x}_s^i$ and adjoint states $\overline{p}_s^i$ satisfying~\eqref{eq:stateEquationDiscreteApp} and \eqref{eq:adjointEquationDiscreteApp}.
We further assume that $\nabla \widetilde{f}(\overline{x}_s^i,\overline{T},\overline{\theta},z^i)$ has full rank for all $i=1,\ldots,N$ and $s=0,\ldots,S$.
Then, we have
\begin{equation}
-\frac{1}{N}\sum_{s=0}^{S-1}\sum_{i=1}^N\langle\overline{p}_{s+1}^i,B(\overline{T})^{-1}(\overline{x}_{s+1}^i-\overline{x}_s^i)\rangle=0.
\label{eq:optimalTDiscreteApp}
\end{equation}
\end{theorem}
\begin{proof}
We compute the derivative of~$\widetilde{f}(x_s^i,T,\theta,z^i)$ with respect to~$T$:
\begin{align}
\nabla_2\widetilde{f}(x_s^i,T,\theta,z^i)
=&\tfrac{1}{S}B(T)^{-1}\Big(A^\top z-\nabla_1\mathcal{R}(x_s^i,\theta)
-A^\top AB(T)^{-1}(x_s^i+\tfrac{T}{S}(A^\top z-\nabla_1\mathcal{R}(x_s^i,\theta)))\Big)\notag\\
=&\frac{1}{T}B(T)^{-1}(x_{s+1}^i-x_s^i),
\end{align}
where we use
\begin{equation}
\tfrac{\dx}{\dx T}(B(T)^{-1})=-B(T)^{-1}\left(\tfrac{\dx}{\dx T}B(T)\right)B(T)^{-1}.
\end{equation}
The existence of the Lagrange multiplier~$p$ is readily implied by the rank condition of~$\nabla\widetilde{f}$ \cite[Chapter~1]{ItKu08}.
At the optimum $(\overline{x},\overline{T},\overline{\theta},z,\overline{p})$ the Lagrange functional~$L_S$ 
satisfies the following optimality condition with respect to~$T$:
\begin{equation}
\nabla_2 L_S(\overline{x},\overline{T},\overline{\theta},z,\overline{p})
=-\sum_{s=0}^{S-1}\sum_{i=1}^N\langle\overline{p}_{s+1}^i,\nabla_2\widetilde{f}(\overline{x}_s^i,\overline{T},\overline{\theta},z^i)\rangle
=-\sum_{s=0}^{S-1}\sum_{i=1}^N\langle\overline{p}_{s+1}^i,\tfrac{1}{\overline{T}}B(\overline{T})^{-1}(\overline{x}_{s+1}^i-\overline{x}_s^i)\rangle,
\end{equation}
which readily implies~\eqref{eq:optimalTDiscreteApp}.
\end{proof}
In the case of image denoising, the scaling $B(\overline{T})$ in the optimality condition~\eqref{eq:optimalTDiscreteApp} is neglected.

An important property of any learning based method is the dependency of the learned parameters with respect to different training datasets~\cite{GeBi92}.
\begin{theorem}
Let $(T,\theta),(\widetilde{T},\widetilde{\theta})$ be two pairs of control parameters obtained from two different training datasets.
We denote by~$x,\widetilde x\in(\R^{nC})^{(S+1)}$ two solutions of the state equation with the same observed data~$z$ and initial condition~$x_\init$, i.e.
\begin{equation}
x_{s+1}=\widetilde{f}(x_s,T,\theta,z),\quad
\widetilde{x}_{s+1}=\widetilde{f}(\widetilde{x}_s,\widetilde{T},\widetilde{\theta},z)
\end{equation}
for $s=1,\ldots,S-1$ and $x_0=\widetilde{x}_0=x_\init$.
Let $L_\mathcal{R}$ be the Lipschitz constant of~$\mathcal{R}$, i.e.
\begin{equation}
\Vert\nabla_1\mathcal{R}(x,\theta)-\nabla_1\mathcal{R}(\widetilde{x},\widetilde{\theta})\Vert_2\leq L_\mathcal{R}
\left\Vert
\begin{pmatrix}
x\\
\theta
\end{pmatrix}
-
\begin{pmatrix}
\widetilde{x}\\
\widetilde{\theta}
\end{pmatrix}
\right\Vert_2
\end{equation}
for all~$x,\widetilde{x}\in\R^{nC}$ and all~$\theta,\widetilde{\theta}\in\Theta$.
Then,
\begin{align}
\Vert x_{s+1}-\widetilde{x}_{s+1}\Vert_2
\leq&\Vert B(T)^{-1}-B(\widetilde{T})^{-1}\Vert_2
\Big(\Vert x_s\Vert_2+\tfrac{T}{S}\Vert A^\top z\Vert_2\notag\\
&+\tfrac{T}{S}\Vert\nabla_1\mathcal{R}(x_s,\theta)\Vert_2\Big)+\Vert B(\widetilde{T})^{-1}\Vert_2
\Big(\Vert x_s-\widetilde{x}_s\Vert_2\notag\\
&+\tfrac{\vert T-\widetilde{T}\vert}{S}\Vert A^\top z\Vert_2+\tfrac{\vert T-\widetilde{T}\vert}{S}\Vert\nabla_1\mathcal{R}(\widetilde{x}_s,\widetilde{\theta})\Vert_2
+\tfrac{T}{S}L_\mathcal{R}\Vert(x,\theta)^\top-(\widetilde{x},\widetilde{\theta})^\top\Vert_2\Big).
\end{align}
\end{theorem}
\begin{proof}
Let $g(x,T,\theta)\coloneqq x+\frac{T}{S}A^\top z-\frac{T}{S}\nabla_1\mathcal{R}(x,\theta)$.
By definition, 
\begin{align}
x_{s+1}&=B(T)^{-1}g(x_s,T,\theta),\notag\\
\widetilde{x}_{s+1}&=B(\widetilde{T})^{-1}g(\widetilde{x}_s,\widetilde{T},\widetilde{\theta}).
\end{align}
Thus, we can estimate the upper bound as follows:
\begin{align}
\Vert x_{s+1}-\widetilde{x}_{s+1}\Vert_2
=&\Vert B(T)^{-1}g(x_s,T,\theta)-B(\widetilde{T})^{-1}g(\widetilde{x}_s,\widetilde{T},\widetilde{\theta})\Vert_2\notag\\
\leq&\Vert B(T)^{-1}-B(\widetilde{T})^{-1}\Vert_2
\Big(\Vert x_s\Vert_2+\tfrac{T}{S}\Vert A^\top z\Vert_2
+\tfrac{T}{S}\Vert\nabla_1\mathcal{R}(x_s,\theta)\Vert_2\Big)+\Vert B(\widetilde{T})^{-1}\Vert_2
\Big(\Vert x_s-\widetilde{x}_s\Vert_2\notag\\
&+\tfrac{\vert T-\widetilde{T}\vert}{S}\Vert A^\top z\Vert_2+\tfrac{\vert T-\widetilde{T}\vert}{S}\Vert\nabla_1\mathcal{R}(\widetilde{x}_s,\widetilde{\theta})\Vert_2
+\tfrac{T}{S}L_\mathcal{R}\Vert(x_s,\theta)^\top-(\widetilde{x}_s,\widetilde{\theta})^\top\Vert_2\Big).\notag
\end{align}
This concludes the proof.
\end{proof}
Hence, this theorem provides a computable upper bound for the norm difference of two states with observed value~$z$ and initial value~$x_\init$ evaluated at the same step~$s$,
which amounts to a sensitivity analysis w.r.t.~the training data.

\section{Additional numerical results}
Figure~\ref{fig:denoising} depicts four additional results for the image denoising task computed with TDV$_{25}^3$ for a noise level $\sigma=25$.
The best results in terms of the PSNR value are obtained for $S=10$ and are highlighted in red.

\begin{figure}
\centering
\includegraphics[width=\linewidth]{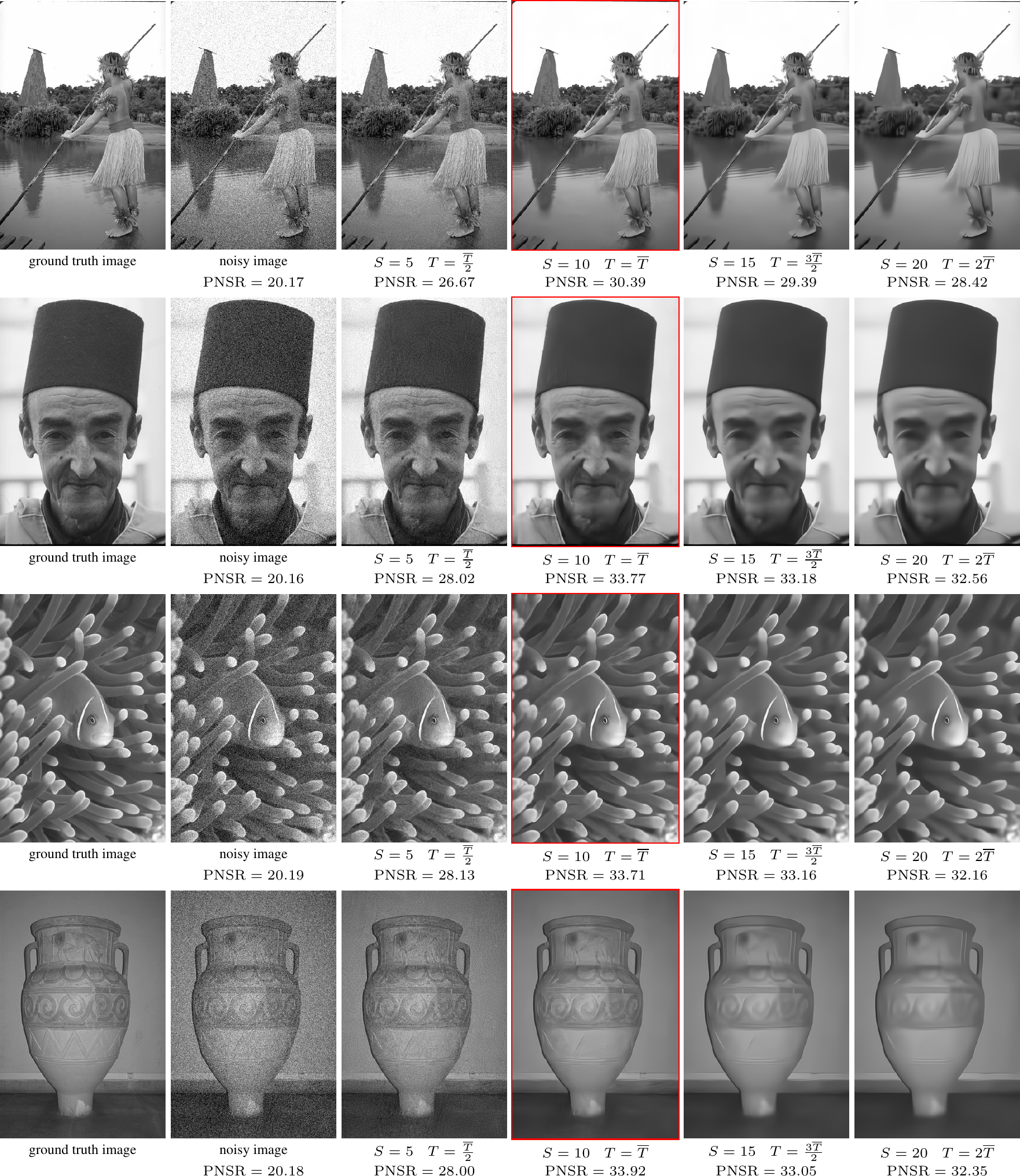}
\caption{From left to right: Ground truth, noisy input with noise level~$\sigma=25$ and resulting output of TDV$^3$ for $(S,T)\in\{(5,\frac{\overline{T}}{2}),(10,\overline{T}),(15,\frac{3\overline{T}}{2}),(20,2\overline{T})\}$,
where the optimal stopping time is $\overline{T}=0.0297$.
Note that the best image is framed in red.}
\label{fig:denoising}
\end{figure}

Figure~\ref{fig:sr} shows four image sequences associated with image super-resolution with a scale factor~$4$.
The algorithm generates a transition from the initial upsampled ($S=0$) to a cartoon-like image with strongly pronounced edges ($S=20$).
Here, the optimal stopping time is image dependent and determines the best image in terms of PSNR value.

\begin{figure}
\centering
\includegraphics[width=\linewidth]{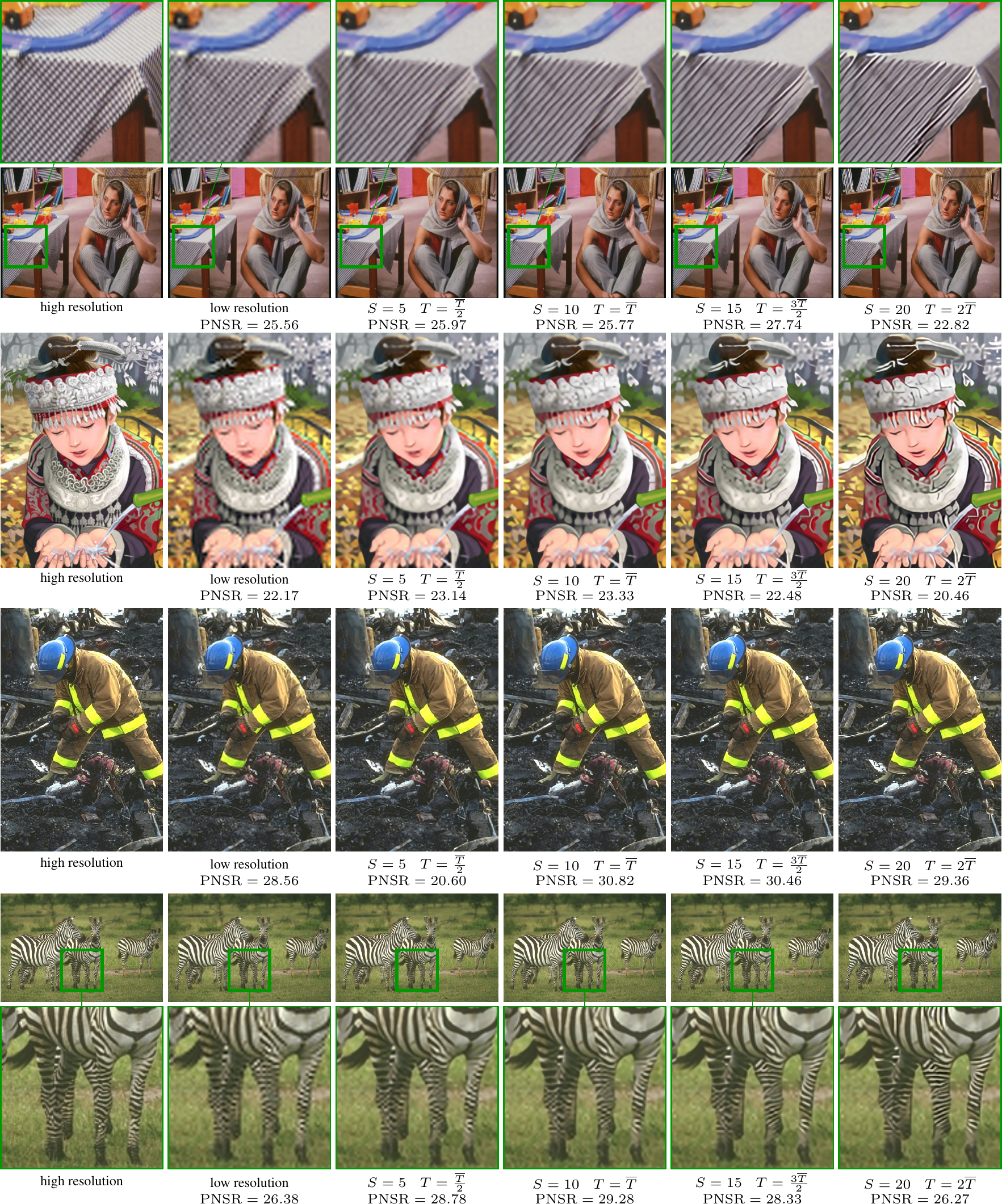}
\caption{
From left to right: High resolution, low resolution with scale factor~$4$, and resulting output of TDV$^3$ for $(S,T)\in\{(5,\frac{\overline{T}}{2}),(10,\overline{T}),(15,\frac{3\overline{T}}{2}),(20,2\overline{T})\}$,
where the optimal stopping time is $\overline{T}=0.098$.
}
\label{fig:sr}
\end{figure}

\section{Transferring TDV to different linear inverse problems}
In this section, we elaborate on how the prior information incorporated in the TDV regularizer can be transferred to different linear inverse problems \emph{without} any adaption of the learnable parameters.

As before, we consider minimizing the variational energy
\begin{equation}
\mathcal{E}(x,\theta,z)\coloneqq\mathcal{D}(x,z)+\mathcal{R}(x,\theta),
\label{eq:energy}
\end{equation}
in which the problem specific data fidelity term 
\begin{equation}
\mathcal{D}(x,z) = \frac{\lambda}{2}\Vert Ax-z\Vert_2^2
\end{equation}
is used to estimate an image~$x\in\R^{nC}$ from observed measurements~$z\in\R^{lC}$.
The linear operator~$A\in\R^{lC\times nC}$ models the data acquisition process of the inverse problem at hand.
The scalar weight~$\lambda>0$ allows to manually balance the weighting between data fidelity and regularization.
To incorporate prior knowledge in the reconstruction process~\eqref{eq:energy}, we consider the smooth TDV regularizer~$\mathcal{R}(x,\theta)$.
Here, we assume that the TDV parameters~$\theta$ have already been determined by a previous training process for a different inverse problem (e.g. image denoising) as presented in the paper.
As a result, we are able to transfer the regularization strategies obtained by training from data on a different task to another specific restoration/reconstruction problem by just using the learned regularizer.

For given measurements~$z\in\R^{lC}$ and previously computed trainable parameters~$\theta\in\Theta$ we retrieve the restored image~$x$ via
\begin{equation}
\argmin_{x\in\R^{nC}}\mathcal{E}(x,\theta,z).
\label{eq:min}
\end{equation}
Since both the data fidelity term and the regularization term are smooth, any convergent first-order gradient method is suitable for optimizing~\eqref{eq:min}.
Algorithm~\ref{algo:agdlb} lists the applied accelerated gradient descent method with Lipschitz backtracking.
We adopt a Lipschitz backtracking strategy to approximate the local Lipschitz constant.
Algorithm~\ref{algo:agdlb} is known to converge to stationary points of~\eqref{eq:min} if the number of iterations~$K$ is chosen large enough~\cite{ChPo16,PoSa16}.
\begin{algorithm}[t]
\SetInd{1ex}{1ex}
\For{$k=1$ \KwTo $K$}{
\tcc{over-relaxation}
$\widehat{x}_k=x_k+\tfrac{1}{\sqrt{2}}(x_k-x_{k-1})$\;
\While{True}{
$x_{k+1}=\widehat{x}_k-\frac{1}{L}\nabla_1\mathcal{E}(\widehat{x}_k,\theta,z)$\;
\tcc{Lipschitz backtracking}
$q\!=\!\langle x_{k+1}-\widehat{x}_k,\nabla\mathcal{E}(\widehat{x}_k,\theta,z)\rangle+\frac{L}{2}\Vert x_{k+1}-\widehat{x}_k\Vert_2^2$\;
\If{$\mathcal{E}(x_{k+1},\theta,z)\leq\mathcal{E}(\widehat{x}_k,\theta,z)+q$}{
$L=\frac{L}{2}$\;
\Break\;
}
\Else{
$L=2L$\;
}
}
}
\caption{Accelerated gradient descent with Lipschitz backtracking.}
\label{algo:agdlb}
\end{algorithm}

To sum up, the idea of reusing prior knowledge extracted from data for a different inverse problem by means of a regularizer enables a simple and effective approach that does not require any further training.
Next, we demonstrate the applicability of this idea by applying a regularizer learned for gray-scale image denoising to computed tomography (CT) and magnetic resonance imaging (MRI) reconstruction.

\section{Computed tomography}
To showcase the effectiveness of the transfer idea introduced in the previous section, we consider the linear inverse problem of angular undersampled two dimensional computed tomography (CT).
The task of CT is to reconstruct an image given projection measurements for different acquisition angles.
All these measurements are stacked along the angular dimension to form the sinogram~$z$.
In practice typically $R=2304$ acquisition angles with $768$ projections are considered, thus $z\in\R^{768R}$.
In the case of angular undersampled CT~\cite{ChTa08} only a fraction of the acquisition angles are measured.
Here, the goal is to reconstruct the image with a similar quality by incorporating prior knowledge in the form of regularization.

The angular undersampled CT problem can be addressed by considering the data fidelity term
\begin{equation}
\mathcal{D}(x,z)=\frac{\lambda}{2}\Vert A_Rx - z\Vert_2^2,
\end{equation}
where $A_R: \R^{R\cdot768\times768\cdot768}$ is the linear projection operator acquiring $R$ angles introduced in~\cite{HaMu18} and $z=A_Ry$ is the measured sinogram of the true image~$y$.
We use the TDV$_{25}^3$ regularizer trained for image denoising and $S=10$ without changing its learnable parameters.

We evaluate the qualitative performance of the proposed regularizer transfer approach on a sample slice of the MAYO dataset~\cite{McBa17}.
The reconstructed images from the fully-sampled, 4/8-fold undersampled sinogram are depicted in Figure~8 in the paper.
It is clearly visible that the TDV$_{25}^3$ regularizer yields reconstructed images of similar quality.
Consequently, the regularizer is able to remove the undersampling artifacts while preserving fine details.

\section{Magnetic resonance imaging}
In this section, we demonstrate the wide-ranging usability of the proposed approach by transferring the TDV regularizer trained for image denoising
to parallel accelerated MRI, which exhibits strong and structured undersampling artifacts.

In detail, in parallel accelerated MRI only a subset of the k-space data of each coil is measured during the acquisition process (for details see~\cite{HaKl18}).
The task in parallel accelerated MRI is the reconstruction of the scanned image~$x$ given the undersampled k-space data~$z_i$ of each receiver coil~$i=1,\ldots,N_C$. 
To account for multiple coils, we change the data fidelity term to
\begin{equation}
\mathcal{D}(x,\{z\}_{i=1}^{N_C})=\frac{\lambda}{2}\sum_{i=1}^{N_C}\Vert M_RFC_ix - z_i\Vert_2^2.
\end{equation}
The linear operator~$C_i\in\C^{n\times n}$ weights the complex image estimate by the coil sensitivity map of the $i^{th}$ coil, which was computed by~\cite{UeLa14}, $F\in\C^{n\times n}$ is the Fourier transform, and $M_R\in\C^{n\times n}$ is an binary mask for $R$-fold undersampling.
Here, we consider Cartesian undersampled k-space data as in~\cite{HaKl18}.
As in the previous sections, the scalar weight~$\lambda$ needs to be adapted to balance data fidelity and regularization.
Despite the large difference between noise artifacts and the undersampling artifacts introduced by Cartesian undersampling, we use the TDV$_{25}^3$ regularizer trained for image denoising and $S=10$ to regularize this linear inverse problem.

We perform a qualitative evaluation of the proposed approach on a representative slice of an undersampled MRI knee image.
The slice has a resolution of $n=320\cdot320$ pixels, and $N_C=15$ receiver coils were used during the acquisition.
Figure~9 in the paper depicts the results of the accelerated parallel MRI problem using the TDV$_{25}^3$ regularizer.
Although the TDV regularizer was not trained to account for undersampling artifacts, all these artifacts are removed in the reconstructions and only some details in the bone are lost.
This highlights the versatility and effectiveness of the proposed TDV regularizer since both challenging medical inverse problems can be properly addressed \emph{without} any fine-tuning of the learned parameters.

{\small
\bibliographystyle{alpha}
\bibliography{references}
}

\end{document}